\documentclass[12pt,a4paper]{amsart}
\usepackage[latin1]{inputenc}
\usepackage[margin = 26mm]{geometry}
\usepackage{amsmath, amssymb, amsthm}
\usepackage{amsfonts,dsfont}
\usepackage{amssymb, relsize}
\usepackage{graphicx,overpic}
\usepackage{xcolor}

\usepackage{mathpazo} 
\usepackage[scaled]{helvet} %

\normalfont

\usepackage[T1]{fontenc}
\usepackage{subcaption}

\usepackage[colorlinks=true,linkcolor=red]{hyperref}

\newtheorem{theorem}{Theorem}[section]
\newtheorem{lemma}[theorem]{Lemma}
\newtheorem{proposition}[theorem]{Proposition}
\newtheorem{corollary}[theorem]{Corollary}

\newtheorem*{theorem*}{Theorem}

\newtheorem{innerthmrep}{Theorem}
\newenvironment{thmrep}[1]
  {\renewcommand\theinnerthmrep{#1}\innerthmrep}
  {\endinnerthmrep}

\usepackage{color}

\newcommand{\ts}{\textsuperscript}

\newcommand{\Hom}{\mathrm{Hom}}
\newcommand{\ind}{\mathds{1}}

\title{Statistics of finite degree covers of torus knot complements}
\author{Elizabeth Baker}
\address{T{\"u}bingen, Germany \\
\tt elizabeth.baker@gmx.net}
\author{Bram Petri}
\address{Institut de Math\'ematiques de Jussieu-Paris Rive Gauche \\
 Sorbonne Universit\'e \\
 Paris, France \\
\tt bram.petri@imj-prg.fr}

\date{\today}

\begin{document}

\begin{abstract}
In the first part of this paper, we determine the asymptotic subgroup growth of the fundamental group of a torus knot complement. In the second part, we use this to study random finite degree covers of torus knot complements. We determine their Benjamini--Schramm limit and the linear growth rate of the Betti numbers of these covers. All these results generalise to a larger class of lattices in $\mathrm{PSL}(2,\mathbb{R})\times \mathbb{R}$. As a by-product of our proofs, we obtain analogous limit theorems for high index random subgroups of non-uniform Fuchsian lattices with torsion.
\end{abstract}

\maketitle
	
\section{Introduction}

A classical theorem due to Hempel \cite{Hempel} states that the fundamental group of a tame $3$-manifold is residually finite. As such, it has many finite index subgroups, or equivalently, the manifold has lots of finite degree covers. 

In this paper we study the fundamental groups of torus knot complements and groups closely related to these. We ask two questions: How fast does the number of index $n$ subgroups grow as a function of $n$? And what are the properties of a random index $n$ subgroup and the corresponding degree $n$ cover?

We will study groups of the form
\[ \Gamma_{p_1,\ldots, p_m} = \langle x_1, \ldots x_m| \; x_1^{p_1} = x_2^{p_2} = \cdots = x_m^{p_m} \rangle. \]
When $\gcd(p,q)=1$ and $p,q\geq 2$ then $\Gamma_{p,q}$ is the fundamental group of a $(p,q)$-torus knot complement. More generally, $\Gamma_{p_1,\ldots,p_m}$ is a central extension of the form
\begin{equation}\label{eq_central_ext}
 1\longrightarrow \mathbb{Z}  \longrightarrow \Gamma_{p_1,\ldots,p_m} \stackrel{\Phi_{p_1,\ldots,p_m}}{\longrightarrow} C_{p_1} * \cdots * C_{p_m} \longrightarrow 1,
 \end{equation}
where  $C_p$ denotes the finite cyclic group of order $p$ and $\Phi_{p_1,\ldots,p_m}$ is the map that sends the generator $x_j$ to a generator of $C_{p_j}$. 

We will consider the case $\sum_{j=1}^m \frac{1}{p_j} < m-1$, which includes all torus knot groups. In this setting $\Gamma_{p_1,\ldots,p_m}$ appears as a non-uniform lattice in $\mathrm{PSL}(2,\mathbb{R})\times \mathbb{R}$  (see for instance \cite[Proposition 7.2]{Eckmann}). $\Phi_{p_1,\ldots,p_m}(\Gamma_{p_1,\ldots,p_m})$ is then the projection onto $\mathrm{PSL}(2,\mathbb{R})$ of $\Gamma_{p_1,\ldots,p_m}$ and is a Fuchsian group that acts on the hyperbolic plane $\mathbb{H}^2$. The orbifold Euler characteristic of $\Phi_{p_1,\ldots,p_m}(\Gamma_{p_1,\ldots,p_m}) \backslash \mathbb{H}^2$ is 
\[
-m+1+\sum_{i=1}^m\frac{1}{p_i},
\] 
a number that will appear repeatedly throughout this paper.

\subsection{Subgroup growth}
The first of our questions above asks for the \emph{subgroup growth} of these groups. That is, the study of the large $n$ behaviour of the quantity $a_n(\Gamma_{p_1,\ldots,p_m})$ -- the number of index $n$ subgroups of $\Gamma_{p_1,\ldots,p_m}$.

Because $\Gamma_{p_1,\ldots,p_m}$ is a central extension of a free product of finite groups by $\mathbb{Z}$, a standard argument which we provide in Section \ref{sec_subgroupgrowth}, combined with results due to M\"uller \cite{Muller_FreeProducts}, Volynets \cite{Volynets} and Wilf \cite{Wilf} on the subgroup growth on free products of finite groups proves the following:
\begin{theorem}\label{thm_main1}
	Let $p_1,\dots, p_m \in \mathbb{N}_{>1}$ such that $\sum_{j=1}^m \frac{1}{p_j} < m-1$. Then it holds that\footnote{Here and throughout the paper, the notation $f(n)\sim g(n)$ as $n\to \infty$ will indicate that $\lim_{n\to\infty} f(n)/g(n) \to 1$.}
	$$a_n(\Gamma_{p_1,\ldots, p_m}) \sim A_{p_1,\ldots,p_m} \cdot n^{-1/2}\cdot \exp\left( \sum_{i=1}^m\sum_{\substack{0 < j < p_i \\ \text{s.t. } j|p_i}} \frac{n^{j/p_i}}{j} \right) \cdot \left(\frac{n}{e}\right)^{n\cdot\left(m-1-\sum_{i=1}^m \frac{1}{p_i}\right)} 
$$
as $n\to\infty$, where
$$
A_{p_1.\ldots p_m} = \sqrt{2\pi} \exp\left( - \sum_{i:\;p_i \text{ even}} \frac{1}{2p_i}\right) \prod_{i=1}^m p_i^{-1/2}.
$$
\end{theorem}


The theorem above also generalizes to free products of the form 
\[\Gamma_{p_{1,1},\ldots, p_{1,m_1}} * \cdots * \Gamma_{p_{r,1}\ldots,p_{r,m_r}}\]
where $\sum_j p_{i,m_i} < m_i -1$ for all $i=1,\ldots, r$. In topological terms, this corresponds to taking connected sums. 

Using the asymptotic behaviour of $a_n(\Gamma_{p_1,\ldots,p_m})$, one can also derive an asymptotic equivalent for the related sequence
\[ h_n(\Gamma_{p_1,\ldots,p_m}) = \left| \Hom(\Gamma_{p_1,\ldots,p_m},S_n) \right|. \]
First of all, note that \eqref{eq_central_ext} implies that
\begin{equation}\label{eq_hn_projection}
 h_n(\Gamma_{p_1,\ldots,p_m}) \quad \geq \quad  h_n(C_{p_1}*\cdots * C_{p_m}) = \prod_{j=1}^m h_n(C_{p_j},S_n). 
 \end{equation}
It turns out that asymptotically the bound in \eqref{eq_hn_projection} is tight. In other words, a typical homomorphism factors through $\Phi_{p_1,\ldots,p_m}$. We prove:
\begin{theorem}\label{thm_main2}
$p_1,\dots, p_m \in \mathbb{N}_{>1}$ such that $\sum_{j=1}^m \frac{1}{p_j} < m-1$. Then
\[ \frac{\left|\left\{\rho \in \Hom(\Gamma_{p_1,\ldots,p_m},S_n);\; \rho \text{ factors through }\Phi_{p_1,\ldots,p_m} \right\}\right|}{h_n(\Gamma_{p_1,\ldots,p_m})} \longrightarrow 1  \]
as $n\to\infty$.
\end{theorem}

The analogous result is also known to hold for orientable circle bundles over surfaces \cite{LisMed}. A similar observation was also made for certain sequences of amalgamated products in \cite[p. 130]{Muller_FreeProducts}.

As we noted above, we use results due to M\"uller, Volynets and Wilf in order to prove the two results above. The fact that $\Gamma_{p_1,\ldots,p_m}$ has a relatively simple presentation also allows one to derive the theorems above with elementary arguments. In particular, one can write down a closed formula for $h_n(\Gamma_{p_1,\ldots,p_m})$ and derive the theorems from there, using the same type of techniques that M\"uller employed in \cite{Muller_FreeProducts}. We will sketch these arguments, which we employed in a previous version of this paper, in Appendix \ref{app_elt_approach}.

\subsection{Random subgroups and covers}

In the second part of our paper, we use the results above to study random finite index subgroups of $\Gamma_{p_1,\ldots,p_m}$. That is, since the number of index $n$ subgroups of $\Gamma_{p_1,\ldots,p_m}$ is finite, we can pick one uniformly at random and ask for its properties. 

Let us denote our random index $n$ subgroup by $H_n$. This is an example of an Invariant Random Subgroup (IRS) -- i.e. a conjugation invariant Borel measure on the Chabauty space of subgroups of $\Gamma_{p_1,\ldots,p_m}$ (for more details see Section \ref{sec_IRS}).

We will also fix a classifying space $X_{p_1,\ldots,p_m}$  for $\Gamma_{p_1,\ldots,p_m}$. For instance, if $p,q\geq 2 $ and $\gcd(p,q)=1$ we can take the corresponding torus knot complement. More generally, since $\Gamma_{p_1,\ldots,p_m}$ appears as a torsion-free lattice in $\mathrm{PSL}(2,\mathbb{R})\times \mathbb{R}$, we may take the manifold $\Gamma \backslash (\mathbb{H}^2\times \mathbb{R})$. $H_n$ gives rise to a random degree $n$ cover of $X_{p_1,\ldots,p_m}$.

We will study three (related) problems:
\begin{itemize}
\item First, we will ask, given a conjugacy class $K\subset \Gamma_{p_1,\ldots,p_m}$, how many conjugacy classes of $H_n$ the set $K\cap H_n$ contains. We will denote this number by $Z_K(H_n)$. In topological terms, $K$ corresponds to a free homotopy class of loops in $X_{p_1,\ldots,p_m}$. $Z_K(H_n)$ is the number of closed lifts of that loop to the cover of $X_{p_1,\ldots,p_m}$ corresponding to $H_n$. We note that we count these lifts as loops and not as sets. In particular, if the corresponding element in $\Gamma_{p_1,\ldots,p_m}$ is non-primitive\footnote{In this article, \emph{primitive} means ``not a non-trivial power''.}, some of these different lifts may overlap.
\item After this we will ask what IRS the random subgroup $H_n$ converges to as $n\to\infty$. In topological terms, this asks for the \emph{Benjamini--Schramm} limit of the corresponding random cover of $X_{p_1,\ldots,p_m}$ (see Section \ref{sec_BS_convergence} for a definition of Benjamini--Schramm convergence).
\item Finally, we will study the asymptotic behaviour of the real Betti numbers $b_k(H_n;\mathbb{R})$ of $H_n$, or equivalently of the corresponding random cover of  $X_{p_1,\ldots,p_m}$.
\end{itemize}

We will write
\[ L_{p_1,\ldots,p_m} := \ker(\Phi_{p_1,\ldots,p_m}) \simeq \mathbb{Z}. \]
For a torus knot this is the subgroup generated by the longitude. Since $L_{p_1,\ldots,p_m}$ is normal in $\Gamma_{p_1,\ldots,p_m}$, it's also an IRS. 

In the theorem below, $\mathcal{N}(0,1)$ will denote a standard normal distribution on $\mathbb{R}$ and more generally, $\mathcal{N}^{\otimes r}$ a standard normal distribution on $\mathbb{R}^r$. If $K_1,K_2\subset \Gamma$ are conjugacy classes in a group $\Gamma$ and $g\in \Gamma$ and $k_1,k_2 \in \mathbb{N}$ are such that $g^{k_1} \in K_1$ and $g^{k_2} \in K_2$ we will say $K_1$ and $K_2$ have a common root. Note that two conjugacy classes that don't have a common root are in particular distinct.

\pagebreak
We will prove:
\begin{theorem}\label{thm_main3} Let $p_1,\ldots,p_m \in\mathbb{N}_{>1}$ be such that $\sum_{j=1}^m \frac{1}{p_j}<m-1$.
\begin{itemize}
\item[(a)] Let $K_1,\ldots,K_r\subset \Gamma_{p_1,\ldots,p_m}$ be non-trivial conjugacy classes of which no pair have a common root.
\begin{enumerate}
\item If for all $g\in K_i$, for all $i=1,\ldots,r$, the image $\Phi_{p_1,\ldots,p_m}(g)$ is either trivial or of infinite order, then, as $n\to\infty$, the random variables $Z_{K_i}(H_n)$, $i=1,\ldots,r$ are asymptotically independent. Moreover,
\begin{itemize}
\item if $K_i \subset L_{p_1,\ldots,p_m}$ then
$$
\lim_{n\to\infty} \mathbb{P}[Z_{K_i}(H_n) = n] =1
$$
\item and if $K_i \not\subset L_{p_1,\ldots,p_m}$ is the conjugacy class of the $k^{th}$ power of a primitive element then $Z_{K_i}(H_n)$ converges in distribution to a random variable
$$
Z_{K_i}^\infty \sim \sum_{d|k}\; d \cdot X_{1/d},
$$
where $X_{1/d} \sim \mathrm{Poisson}(1/d)$ and $X_1,\ldots,X_{1/k}$ are independent.
\end{itemize}
\item If the images of the elements of $K_i$ under $\Phi_{p_1,\ldots,p_m}$ have order $k_i\in\mathbb{N}$ for $i=1,\ldots,r$, then the vector of random variables
$$
\left(\frac{Z_{K_1}(H_n)-n^{1/k_1}}{\sqrt{l_1}\cdot n^{1/2k_1}},\ldots, \frac{Z_{K_r}(H_n)-n^{1/k_r}}{\sqrt{l_r}\cdot n^{1/2k_r}} \right)
$$
converges in distribution to a $\mathcal{N}(0,1)^{\otimes r}$-distributed random variable as $n\to \infty$. Here $l_i\in\mathbb{N}$ is such that $\Phi_{p_1,\ldots,p_m}(K_i)$ is the conjugacy class of $x_{j_i}^{l_i}$, for $i=1,\ldots,r$.
\end{enumerate}
\item[(b)] As $n\to\infty$, $H_n$ converges to $L_{p_1,\ldots,p_m}$ as an IRS.
\item[(c)] We have that 
$$
\lim_{n\to\infty} \frac{b_k(H_n;\mathbb{R})}{n} =  \left\{
\begin{array}{ll}
m-1 - \sum\limits_{i=1}^m \frac{1}{p_i} & \text{if } k = 1,2 \\
0 & \text{otherwise}.
\end{array}
\right.
$$
in probability.
\end{itemize}
\end{theorem}

Recall that a random variable $X:\Omega\to\mathbb{N}$ is Poisson-distributed with parameter $\lambda>0$ if and only if
$$
\mathbb{P}[X=k] = \frac{\lambda^k e^{-\lambda}}{k!} \quad \forall k \in \mathbb{N}.
$$
So (a) above gives us an explicit limit for the probability that a fixed curve lifts to any given number of curves in the cover. For example, if we denote the random degree $n$ cover of our $(p,q)$-torus knot complement by $X_{p,q}(n)$ and $\gamma$ is any primitive free homotopy class of closed curves in $X_{p,q}(1)$ that is not freely homotopic to the longitude we obtain:
$$
\lim_{n\to\infty}\mathbb{P}[ \gamma \text{ lifts to exactly }3\text{ closed curves in }X_{p,q}(n) ] = \frac{1}{6e} = 0.0613\ldots .
$$
Using our techniques, the limit of the joint distribution of the variables associated to distinct conjugacy classes that do have common roots can also be determined. The independence will then however be lost (see Theorem \ref{thm_factorialmomentsY} below).

(b) in particular implies that a random degree $n$ cover of a torus knot complement does \emph{not} converge to the universal cover of the given torus knot complement as $n\to\infty$. This is different from the behaviour of random finite covers of graphs \cite{DJPP}, surfaces \cite{MageePuder} and many large volume locally symmetric spaces of higher rank \cite{7sam}, that all do converge to their universal covers.

(c) also has implications for the number of boundary tori in a random cover of a torus knot complement. Indeed, together with ``half lives, half dies'' \cite[Lemma 3.5]{Hatcher_3D}, it also implies that the number of boundary components of a degree $n$ cover is typically at most $\Big(1 - \frac{1}{p}-\frac{1}{q}\Big)\cdot n + o(n)$. 

Because all the results in the theorem above are really about the group $\Gamma_{p_1,\ldots,p_m}$, we can also apply them to random covers of more general spaces $Y_{p_1,\ldots,p_m}$ that have $\Gamma_{p_1.\ldots,p_m}$ as their fundamental group (i.e. without assuming that $Y_{p_1,\ldots,p_m}$ is a classifying space for $\Gamma_{p_1,\ldots,p_m}$). In that case, the random cover Benjamini--Schramm converges to the cover of $Y_{p_1,\ldots,p_m}$ corresponding to $L_{p_1,\ldots,p_m}$ and the normalised Betti numbers converge to the $\ell^2$-Betti numbers of that cover.

Finally, we note that we prove analogous results to Theorem \ref{thm_main3} for random index $n$ subgroups of non-cocompact Fuchsian groups.

\begin{theorem}\label{thm_main4}
Let $\Lambda$ be a non-cocompact Fuchsian group of finite covolume.
Moreover, let $G_n<\Lambda$ denote an index $n$ subgroup, chosen uniformly at random. 
\begin{itemize}
\item[(a)] 
\begin{enumerate}
\item If $K_1,\ldots, K_r \subset \Lambda$ are conjugacy classes of infinite order elements of which no pair have a common root. Then, as $n\to\infty$, the vector of random variables
$$
(Z_{K_1}(G_n),\ldots,Z_{K_r}(G_n))
$$
converges jointly in distribution to a vector 
$$
(Z_{K_1}^\infty,\ldots,Z_{K_r}^\infty):\Omega \to \mathbb{N}^r
$$
of independent random variables, such that if $K_i$ is the conjugacy class of a $k^{th}$ power of a primitive element then
$$
Z_{K_i}^\infty \sim \sum_{d|k}\; d \cdot X_{1/d},
$$
where $X_{1/d} \sim \mathrm{Poisson}(1/d)$ and $X_1,\ldots,X_{1/k}$ are independent.
\item \cite[Lemma 4]{MullerPuchta3} If $K_1,\ldots,K_r\subset \Lambda$ are non-trivial conjugacy classes of which no pair have a common root, whose elements have orders $k_1,\ldots,k_r\in\mathbb{N}$ respectively, then the vector of random variables
$$
\left(\frac{Z_{K_1}(G_n)-n^{1/k_1}}{\sqrt{l_1}\cdot n^{1/2k_1}},\ldots, \frac{Z_{K_r}(G_n)-n^{1/k_r}}{\sqrt{l_r}\cdot n^{1/2k_r}} \right)
$$
converges in distribution to a $\mathcal{N}(0,1)^{\otimes r}$-distributed random variable as $n\to \infty$. Here $l_i\in\mathbb{N}$ is such that $K_i$ is the conjugacy class of $x_{j_i}^{l_i}$, for $i=1,\ldots,r$.
\end{enumerate}
\item[(b)] As $n\to\infty$, $G_n$ converges to the trivial group as an IRS.
\end{itemize}
\end{theorem}

Note that the analogue to Theorem \ref{thm_main3}(c) also holds here. However, a much stronger statement follows directly from results by M\"uller--Schlage-Puchta \cite{MullerPuchta3}.

The case of free groups in the theorem above is very similar to results on cycle counts in random regular graphs in the permutation model (see for instance \cite{DJPP} and also \cite{Bol} for a slightly different model). In the case where $\Lambda$ is a free product of finite cyclic groups and $r=1$, Benaych-Georges \cite{BG1} proved point (1) of item (a) for certain conjugacy classes. 

For surface groups similar results have very recently been proved by Magee--Puder \cite{MageePuder}. The case of cocompact Fuchsian groups with torsion is currently open.

\subsection{The structure of the proofs}

Our proofs start with the subgroup growth. We combine results due to M\"uller \cite{Muller_FreeProducts,Muller}, Volynets \cite{Volynets} and Wilf \cite{Wilf} with a classical bound on the number of index $n$ subgroups of an extension (see Proposition \ref{prop_subgp_quotient}) in order to prove Theorem \ref{thm_main1}. Combining this with the recurrence relating the sequences $(a_n(\Gamma_{p_1,\ldots,p_m}))_n$ and $(h_n(\Gamma_{p_1,\ldots,p_m}))_n$ we then obtain Theorem \ref{thm_main2} as well.

The idea behind the proofs of our results on random subgroups is to first prove the analogous results for random index $n$ subgroups of $C_{p_1}*\cdots *C_{p_m}$ and then use the fact that most index $n$ subgroups of $\Gamma_{p_1,\ldots,p_m}$ come from index $n$ subgroups of $C_{p_1}*\cdots *C_{p_m}$ to upgrade these into results about $\Gamma_{p_1,\ldots,p_m}$.

First, we consider the problem of counting the number of fixed points of an element $g\in C_{p_1}*\cdots * C_{p_m}$ under a random homomorphism $C_{p_1}*\cdots * C_{p_m} \to S_n$. There are two cases to consider:
\begin{itemize}
\item M\"uller--Schlage-Puchta \cite{MullerPuchta2,MullerPuchta3} proved the central limit theorem we need for finite order elements. 
\item If $g$ is of infinite order, its number of fixed points can be approximated by a sum of multiples of Poisson-distributed random variables. For certain conjugacy classes, this was proved by Benaych-Georges in \cite{BG1}. We extend this to all conjugacy classes and also show that the statistics of different conjugacy classes are asymptotically independent. 
\end{itemize}
We prove our Poisson distribution result by estimating the factorial moments of the random variables that count the fixed points of $g$. We note that in all of our estimates, the error terms could be made explicit using the error terms in M\"uller's results \cite{Muller_FreeProducts}. Moreover, the Chen--Stein method (see for instance \cite{AGG,BHT,DJPP,CGS}) might give sharper bounds than the method of moments.

The fact that a conjugacy class $K\subset \Gamma_{p_1,\ldots,p_m}$ typically has very few lifts to $H_n$ if it does not lie in $L_{p_1,\ldots,p_m}$ and typically has $n$ lifts if it does (this is essentially Theorem \ref{thm_main3}(a)), implies that the IRS $H_n$ converges to $L_{p_1,\ldots,p_m}$ (Theorem \ref{thm_main3}(b)). Using results by Elek \cite{Elek} and L\"uck \cite{Lueck}, we then also obtain that the normalised Betti numbers of $H_n$ converge to the $\ell^2$-Betti numbers of the cover of $\widetilde{X}_{p_1,\ldots,p_m}/L_{p_1,\ldots,p_m}$.

In Section \ref{sec_fuchsian}, we sketch how to complete the proof of Theorem \ref{thm_main4}.

Finally, in Appendix \ref{app_elt_approach} we explain an alternative approach to Theorems \ref{thm_main1} and \ref{thm_main2} which goes through computing a closed formula for $h_n(\Gamma_{p_1,\ldots,p_m})$.

\subsection{Notes and references} 
As opposed to the case of $2$-manifolds \cite{Dixon,MullerPuchta,LiebeckShalev}, there are very few $3$-manifolds for which the subgroup growth is well understood. For instance, to the best of our knowledge, there isn't a single hyperbolic $3$-manifold group $\Gamma$ for which  the asymptotic behaviour of $a_n(\Gamma)$ is known. It does follow from largeness of these groups \cite{Agol} that the number 
\[ s_n(G) := \sum_{m\leq n} a_m(G)\]
grows faster than $(n!)^\alpha$ for some $\alpha>0$, but even at the factorial scale, the growth (i.e. the optimal $\alpha$) is not known. In the more general setting of lattices in $\mathrm{PSL}(2,\mathbb{C})$ it's known in one very particular case \cite[Section 2.5.2]{BPR}. One of the difficulties in determining $\alpha$ in general is that for a general hyperbolic $3$-manifold, no proof for a factorial lower bound is known that does not rely on Agol's work. 

For Seifert fibred manifolds a little more is known: the subgroup growth of orientable circle bundles over surfaces was determined by Liskovets and Mednykh \cite{LisMed} and the subgroup growth of Euclidean manifolds can be derived from general results on the subgroup growth of virtually abelian groups \cite{dSMcDS,Sulca}.

One can also ask for the number of distinct isomorphism types of subgroups, in which case even less is known \cite{FPPRR}. 

Finally, results similar to our Theorems \ref{thm_main1} and \ref{thm_main2} are known to hold for Baumslag--Solitar groups \cite{Kelley}.

The geometry of a random cover of a graph is a classical subject in the study of random regular graphs (see for instance \cite{AmitLinial,Friedman,DJPP,Puder}). Moreover, it is known that, as $n\to\infty$, a random $2d$-regular graph sampled uniformly from the set of such graphs on $n$ vertices as a model is contiguous to the model given by a random degree $n$ cover of a wedge of $d$ circles \cite{GJKW,Wor}. In other words, random covers are also a tool that can be used to study other models of random graphs.  

Random covers of manifolds are much less well understood. Of course, random graph covers also give rise to random covers of punctured surfaces, so some of the graph theory results can be transported to this context. Very recently, Magee--Puder \cite{MageePuder} and Magee--Naud--Puder \cite{MageeNaudPuder} studied random covers of closed hyperbolic surfaces. They proved that these covers Benjamini--Schramm converge to the hyperbolic plane and that the spectral gap of their Laplacian is eventually larger than $\frac{3}{16}-\varepsilon$ for all $\varepsilon>0$ (given that his holds for the base surface).

More general random surfaces (see for instance \cite{BM,GPY,Mir,Pet,PT,MP,MRR,BCP,GLST,Shrestha}), random $3$-manifolds (see for instance \cite{DT,Maher,BBGHHKPV,HV}) and random knots (see for instance \cite{EZ,BKLMR}) have recently also received considerable attention.

Invariant Random Subgroups were introduced by Ab\'ert--Glasner--Vir\'ag in \cite{AbertGlasnerVirag}, by Bowen in \cite{Bowen} and under a different name by Vershik in \cite{Vershik}, but had been studied in various guises before (see the references in \cite{AbertGlasnerVirag}). Benjamini--Schramm convergence was introduced for graphs in \cite{BenjaminiSchramm} and for lattices in Lie groups in \cite{7sam}. The fact that Benjamini--Schramm convergence implies convergence of normalised Betti numbers was proved for sequences of simplicial complexes in \cite{Elek}, for sequences of lattices in \cite{7sam} and for sequences of negatively curved Riemannian manifolds in \cite{ABBG}.

\subsection*{Acknowledgement}
We thank Jean Raimbault for useful remarks. Moreover, we are very grateful to two anonymous referees for pointing out errors in previous versions of Theorem \ref{thm_main3}(a) and for suggesting a shorter proof of Theorem \ref{thm_main1}.

\section{Preliminaries}

\subsection{Subgroup growth}\label{sec_prelim_subgroup_growth}

As mentioned in the introduction, our results on subgroup growth are based on the connection between finite index subgroups of a group $G$ and transitive permutation representations of $G$. Indeed, an index $n$ subgroup $H<G$ gives rise to a transitive action of $G$ on the finite set $G/H$ and as such, upon labelling the elements of $G/H$ with the numbers $1,\ldots,n$, a homomorphism $G\to S_n$. Here $S_n$ denotes the symmetric group on $n$ elements. This leads to the following (see \cite[Proposition 1.1.1]{LubotzkySegal} for a detailed proof):

\begin{proposition}\label{subgp_trans}	Let $G$ be a group and $n\in\mathbb{N}$. Then 
$$a_n(G) = \frac{t_n(G)}{(n-1)!},$$
where 
$$t_n(G) =\left|\{\varphi: G \to S_n |\ \varphi(G) \text{ acts transitively on } \{1,...n\} \}\right| . $$
\end{proposition}

We shall also be using \cite[Proposition 1.3.2]{LubotzkySegal}, which states:

\begin{proposition}\label{prop_subgp_quotient}
	Let $G$ be a group, $N$ a normal subgroup of $G$, and $Q = G/N$. Then
	$$a_n(G) \leq \sum_{t|n} a_{n/t}(Q)a_t(N)t^{rk(Q)}.$$
\end{proposition}

In order to count fixed points of finite order elements under random homomorphisms later on, we will need to split homomorphisms $\Gamma \to S_n$ according to their orbits. Given a partition $\pi = (\pi_1,\ldots, \pi_r)$ (a non decreasing sequence of positive integers called the parts of $\pi$), we will write 
$$
h_\pi(\Gamma)=\left|\left\{\rho\in\Hom(\Gamma,S_{|\pi|});\begin{array}{c}\text{the orbits of }\rho(\Gamma)\text{ on }\{1,\ldots,|\pi|\}\\ \text{have sizes }\pi_1,\ldots, \pi_r \end{array}\right\}\right|,
$$
where $|\pi| = \pi_1 +\ldots + \pi_r$. 

Given a group $\Gamma$, we will encode these numbers in an exponential generating function
$$
 F_\Gamma(x,y) = \sum_\pi \frac{h_\pi(\Gamma)}{|\pi|!} \; x^{|\pi|} y^\pi,
$$
 in an infinite number of formal variables $x$, $y_1,y_2,\ldots$, where $y^\pi = \prod_i y_{\pi_i}$.

This generating function can be explicitly computed in terms of the sequence $a_n(\Gamma)$:
\begin{lemma}\label{lem_expprinc} Let $\Gamma$ be a finitely generated group. We have
$$
F_\Gamma(x,y) = \exp\left(\sum_{i=1}^{\infty} \frac{a_i(\Gamma)}{i} x^{i}y_i \right).
$$
\end{lemma}

Another result we will need is on the asymptotic number of homomorphisms $C_m \to S_n$ (or equivalently the number of elements of order $m$ in $S_n$). The result we will use is due to Volynets \cite{Volynets} and independently Wilf \cite{Wilf} and fits into a large body of work, starting with classical results by Chowla--Herstein--Moore \cite{CHM}, Moser--Wyman \cite{MW}, Hayman \cite{Hayman} and Harris--Schoenfeld \cite{HS} and culminating in a paper by M\"uller \cite{Muller} in which the asymptotic behaviour of $h_n(G)$ as $n\to\infty$ is determined for any finite group $G$. It states:

\begin{theorem}[Volynets \cite{Volynets}, Wilf \cite{Wilf}]\label{hn_cyclic} Let $m_1,\ldots,m_k \in \mathbb{N}$. Then
$$
h_n(C_m) \sim A_m \cdot \exp\left( \sum_{\substack{d|m \\ d<m}} \frac{1}{d}n^{d/m}\right) \cdot  \left(\frac{n}{e}\right)^{n \cdot \left(1 - \frac{1}{m}\right)} 
$$
as $n\to\infty$. Here
$$
A_m = \left\{\begin{array}{ll}
 m^{-1/2} ; & m \text{ odd}\\
 m^{-1/2}\exp\left(-\frac{1}{2m}\right); & m \text{ even.} 
   \end{array}\right.
$$
\end{theorem}

Finally, we will need the following result by M\"uller:
\begin{theorem}[M\"uller \cite{Muller_FreeProducts}]\label{thm_muller-transitive}
Let $p_1,\ldots,p_m \in \mathbb{N}_{>1}$ such that $\sum_{i=1}^m \frac{1}{p_i} < m-1$. Then
$$
t_n(C_{p_1}*\cdots * C_{p_m}) \sim h_n(C_{p_1}*\cdots*C_{p_m}) \quad \text{as }n\to\infty.
$$
\end{theorem}
In fact, M\"uller also provides error terms and proves the theorem for more general groups; we refer to his paper for details.

\subsection{Probability theory}

For our Poisson approximation results, we will use the method of moments. Given a random variable $Z:\Omega \to \mathbb{N}$ and $k\in\mathbb{N}$, we will write
$$
(Z)_k = Z(Z-1)\cdots (Z-k+1).
$$
Moreover, recall that a sequence of random variables $Z_n:\Omega_n\to\mathbb{N}^d$ is said to converge \emph{jointly in distribution} to a random variable $Z:\Omega\to\mathbb{N}^d$ if and only if
$$
\mathbb{P}[Z_n\in A] \stackrel{n\to\infty}{\longrightarrow} \mathbb{P}[Z\in A] \quad \forall A\subset \mathbb{N}^d.
$$
The following theorem is classical. For a proof see for instance \cite[Theorem 1.23]{Bol_book}.
\begin{theorem}[The method of moments]\label{thm_moments}
Let $Z_{n,1},Z_{n,2},\ldots,Z_{n,r}:\Omega_n\to\mathbb{N}$, $n\in\mathbb{N}$ be random variables. If there exist $\lambda_1,\ldots,\lambda_r>0$ such that for all $k_1,\ldots,k_r \in \mathbb{N}$
$$
\lim_{n\to\infty}\mathbb{E}\left[(Z_{n,1})_{k_1}(Z_{n,2})_{k_2}\cdots (Z_{n,r})_{k_r}\right] = \lambda_1^{k_1}\lambda_2^{k_2}\cdots \lambda_r^{k_r},
$$
then $(Z_{n,1},\ldots,Z_{n,r}):\Omega_n\to\mathbb{N}^r$ converges jointly in distribution to a vector of random variables $(Z_1,\ldots,Z_r):\Omega\to\mathbb{N}^r$ where
\begin{itemize}
\item $Z_i\sim\mathrm{Poisson}(\lambda_i)$, $i=1,\ldots,r$
\item The random variables $Z_1,\ldots,Z_r$ form an independent family.
\end{itemize}
\end{theorem}

\subsection{Invariant Random Subgroups}\label{sec_IRS}

We will phrase our results on random subgroups in the language of Invariant Random Subgroups. For a finitely generated group $\Gamma$, $\mathrm{Sub}(\Gamma)$ will denote the \emph{Chabauty space} of subgroups of $\Gamma$ (see for instance \cite{Gelander} for an introduction).

We will be interested in random index $n$ subgroups of such a group $\Gamma$. Given $n\in\mathbb{N}$, we will write 
$$
\mathcal{A}_n(\Gamma) = \{H<\Gamma;\; [\Gamma:H] = n\} \subset \mathrm{Sub}(\Gamma),
$$
so that $a_n(\Gamma) = |\mathcal{A}_n(\Gamma)|$. Studying a random index $n$ subgroup of $\Gamma$ comes down to understanding the measure $\mu_n$ on $\mathrm{Sub}(\Gamma)$, defined by
\[\mu_n = \frac{1}{a_n(\Gamma)} \sum_{H \in \mathcal{A}_n(\Gamma)} \delta_H\]
where $\delta_H$ denotes the Dirac mass on $H\in \mathrm{Sub}(\Gamma)$. 

$\mu_n$ is an example of what is called an \emph{Invariant Random Subgroup} (IRS) of $\Gamma$ -- i.e. a Borel probability measure on $\mathrm{Sub}(\Gamma)$ that is invariant under conjugation by $\Gamma$. We will write $\mathrm{IRS}(\Gamma)$ for the space of IRS's of $\Gamma$ endowed with the weak-* topology. This space has been first studied under this name in \cite{AbertGlasnerVirag} and \cite{Bowen} and under a different name in \cite{Vershik}.

We will also use a characterisation for convergence in $\mathrm{IRS}(\Gamma)$ terms of fixed points. This characterisation is probably well known, but we couldn't find the exact statement in the literature (for instance \cite[Lemma 16]{AbertGlasnerVirag} is very similar). We will provide a proof for the sake of completeness. 

Given a function $f:\mathrm{Sub}(\Gamma) \to \mathbb{C}$, we will write $\mu_n(f)$ for the integral of $f$ with respect to $\mu_n$ (all measures considered in our paper are finite sums of Dirac masses, so this is always well defined). Moreover, if $K\subset \Gamma$ is a conjugacy class then we will write 
$$
Z_K: \mathcal{A}_n(\Gamma)\to\mathbb{N}
$$
for the random variable that measures the number of conjugacy classes that $K$ splits into, i.e.
$$
Z_K(H) = | (K\cap H) / H |
$$
where $H$ acts on $K\cap H$ by conjugation. Note that if we fix any $g\in K$ and $\varphi:\Gamma \to S_n$ is a transitive homomorphism corresponding to $H$ (cf. Proposition \ref{subgp_trans}), then
$$
Z_K(H) = |\{ j\in \{1,\ldots,n\};\;\varphi(g)\cdot j = j\}|.
$$

\begin{lemma}\label{lem_IRSfixedpoints}
Let $\Gamma$ be a countable discrete group. Set
$$
\mu_n = \frac{1}{a_n(\Gamma)}  \sum_{\substack{ H\;<\; \Gamma \\ [\Gamma:H] \; = \; n}} \delta_H.
$$
and let $N\vartriangleleft \Gamma$. Then
$$
\mu_n \stackrel{n\to\infty}{\longrightarrow} \delta_{N} \; \text{in }\mathrm{IRS}(\Gamma) \quad \Leftrightarrow \quad \left\{ \begin{array}{ll}
 \mu_n(Z_K) \stackrel{n\to\infty}{=} o(n) & \forall \text{ conjugacy class } K \not\subset N \\
\mu_n(Z_K)\stackrel{n\to\infty}{\sim} n & \forall \text{ conjugacy class } K \subset N
\end{array}
\right.
$$
\end{lemma}

\begin{proof}
We start with the fact that for $g\in K$, $\mu_n(\{H; g\in H\}) = \frac{1}{n}\mu_n(Z_K)$. Indeed, for any $p\in\{1,\ldots,n\}$, the map $\varphi\mapsto \mathrm{Stab}_\varphi\{p\}$ gives an $(n-1)!$-to-$1$ correspondence between transitive homomorphisms $\Gamma\to S_n$ and index $n$ subgroups of $\Gamma$. $Z_K(\varphi)$ equals the number of fixed points of $\varphi(g)$ on $\{1,\ldots,n\}$. As such
\begin{equation}
\mu_n(\{H; g\in H\})    = \frac{1}{n\cdot t_n(\Gamma)} \sum_{p=1}^n \sum_{\varphi \in \mathcal{T}_n(\Gamma)} \ind_{g\in \mathrm{Stab}_\varphi\{p\}}(\varphi)
 = \frac{1}{n}\mu_n(Z_K), \label{eq_mufixedpts}
\end{equation}
where
\[
\mathcal{T}_n(\Gamma) = \left\{ \varphi\in \Hom(\Gamma,S_n);\; \varphi(\Gamma) \curvearrowright \{1,\ldots,n\} \text{ transitively}\right\}.
\]

Now, the topology on $\mathrm{Sub}(\Gamma)$ is generated by sets of the form
$$
O_1(U) := \{ H\in \mathrm{Sub}(\Gamma); \; H\cap U \neq \emptyset \},\quad U \subset\Gamma
$$
and
$$
O_2(V) := \{ H\in \mathrm{Sub}(\Gamma); \; H\cap V = \emptyset \},\quad V \subset\Gamma \text{ finite},
$$
(see for instance \cite{Gelander}). By the Portmanteau theorem, convergence $\mu_n \stackrel{w^*}{\longrightarrow} \delta_N$ is equivalent to
\[
\liminf_{n\to\infty} \mu_n(O) \geq \delta_N(O)
\]
for every open set $O\subset \mathrm{Sub}(\Gamma)$. This is equivalent to proving that $\mu_n(O) \to 1$ for every open set $O\subset \mathrm{Sub}(\Gamma)$ such that $N\in O$. Since every open set is a union of sets of the form $O_1(U)$ and $O_2(V)$, $\mu_n \stackrel{w^*}{\longrightarrow} \delta_N$ if and only if
\begin{equation}\label{eq_portmanteau}
\mu_n(O_1(U)) \to 1 \text{ when }U\cap N \neq \emptyset \quad \text{and} \quad \mu_n(O_2(V)) \to 1 \text{ when }V\cap N = \emptyset
\end{equation}
for all $U\subset \Gamma$ and all finite $V\subset \Gamma$.

Let us first prove that our conditions on the behaviour of $\mu_n(Z_K)$ imply convergence in $\mathrm{IRS}(\Gamma)$. 

We start by checking \eqref{eq_portmanteau} for sets of the form $O_1(U)$. Suppose $g\in  U\cap N$. Using \eqref{eq_mufixedpts} and writing $K$ for the conjugacy class of $g$,
$$
\mu_n(\{H; H\cap U \neq \emptyset\}) \geq \mu_n(\{H; g\in H\})  = \frac{1}{n}\mu_n(Z_K) \to 1,
$$
by our assumption on $\mu_n(Z_K)$. 

Now we deal with sets of the form $O_2(V)$. We will write $K(g)$ for the conjugacy class of an element $g\in\Gamma$. \eqref{eq_mufixedpts} gives us
$$
\mu_n(\{H; H\cap V \neq \emptyset\}) \leq \frac{1}{n} \sum_{g\in V}\mu_n(Z_{K(g)}) \stackrel{n\to\infty}{\longrightarrow} 0,
$$
by our assumptions on $\mu_n(Z_{K(g)})$. This proves the first direction.

For the other direction, suppose $g\in N$ then $\delta_{N}(O_1(\{g\})) = 1$ and hence by \eqref{eq_portmanteau} and \eqref{eq_mufixedpts}, we obtain
$$
\liminf_{n\to\infty} \frac{1}{n}\mu_n(Z_K) = 1,
$$
which proves that $\mu_n(Z_K)\sim n$ as $n\to\infty$. Moreover, if $K$ is a conjugacy class such that $K\not\subset N$ and $g\in K$, then by \eqref{eq_portmanteau} and \eqref{eq_mufixedpts}, 
$$
\liminf_{n\to\infty} \mu_n(O_2(\{g\})) = \liminf_{n\to\infty} 1-\mu_n(O_1(\{g\})) = \liminf_{n\to \infty} 1-\frac{1}{n}\mu_n(Z_K) \geq 1,
$$
which proves that $\mu_n(Z_K) = o(n)$ as $n\to\infty$.
\end{proof}

\subsection{Benjamini--Schramm convergence}\label{sec_BS_convergence}

Now suppose that --- as many of the groups that we study do --- $\Gamma$ admits a finite simplicial complex $X$ as a classifying space. Picking a $0$-cell $x_0\in X$ gives an identification $\Gamma \simeq \pi_1(X,x_0)$. Moreover, an index $n$ subgroup $H<\Gamma$ gives rise to a pointed simplicial covering space
\[ (Y_H,y_H) \to (X,x_0).\]
This means that the measure $\mu_n$ above also gives rise to a probability measure $\nu_n$ on the set
$$
\mathcal{K}_D = \left\{ (Y,y_0); \begin{array}{c}
Y \text{ a connected simplicial complex in} \\
\text{which the degree of }0\text{-cells is at} \\
\text{most }D,\; y_0 \in Y \text{ a }0\text{-cell}
\end{array} \right\} \Bigg/ \sim 
$$ 
for some $D>0$, where two pairs $(Y,y_0)\sim (Y',y_0')$ if there is a simplicial isomorphism $Y\to Y'$ that maps $y_0$ to $y_0'$. This set $\mathcal{K}$ can be metrised by setting
$$
d_\mathcal{K}([Y,y_0],[Y',y_0']) = \frac{1}{1+\sup\left\{R\geq 0;\;\begin{array}{c} \text{The }R\text{-balls around }y_0\text{ and }y_0'\text{ are iso-}\\ \text{morphic as pointed simplicial complexes} \end{array}\right\}}.
$$

This allows us to speak of weak-* convergence of measures on $\mathcal{K}_D$. If there is a pointed simplicial complex $[Z,z_0]\in\mathcal{K}_D$ such that
\[\nu_n \stackrel{w^*}{\longrightarrow} \delta_{[Z,z_0]} \quad \text{as }n\to\infty,\] 
where $\delta_{[Z,z_0]}$ denotes the Dirac mass on $[Z,z_0]$, then we say that the random complex determined by $\nu_n$ \emph{Benjamini--Schramm converges} (or \emph{locally converges}) to $[Z,z_0]$.

We will write $\mathcal{BS}(\mathcal{K}_D)$ for the space of probability measures on $\mathcal{K}_D$ endowed with the weak-* topology. The procedure described above describes a continuous map
$$
\mathrm{IRS}(\Gamma) \to \mathcal{BS}(\mathcal{K}_D),
$$
for some $D>0$, that depends on the choice of classifying space.

\subsection{Betti numbers}

One reason for determining Benjamini--Schramm limits is that they help determine limits of normalised Betti numbers. We will exclusively be dealing with homology with real coefficients in this paper. Given a simplicial complex $X$, we will write
$$
b_k(X) = \dim(H_k(X;\mathbb{R})).
$$

In order to state Elek's result, let $[\beta_1(R),o_1],\ldots, [\beta_M(R),o_M]$ denote all the complexes in $\mathcal{K}_D$ that can appear as an $R$-ball of a complex in $\mathcal{K}_D$. Note that this is a finite list, the length of which depends on $R$ and $D$. Moreover, given a finite simplicial complex $X$ of which all $0$-cells degree at most $D$, we will write
$$
\rho_{\beta_i(R)}(X) = \frac{\left|\left\{x\in V(X);\; \text{The }R\text{-ball around }x\text{ is isomorphic to }\beta_i(R) \right\}\right|}{|V(X)|},\; i=1,\ldots,M
$$
where $V(X)$ denotes the set of $0$-cells of $X$. Elek's theorem now states:

\begin{theorem}[Elek {\cite[Lemma 6.1]{Elek}}]\label{thm_Elek}
Fix $D>0$ and let $(X_n)_n$ be a sequence of finite simplicial complexes in which the degree of every $0$-cell is bounded by $D$. If $|V(X_n)|\to\infty$ and for all $R>0$, for all $i$, $\rho_{\beta_i(R)}(X_n)$ converges as $n\to\infty$, then
$$
\lim_{n\to \infty}\frac{b_k(X_n)}{|V(X_n)|}
$$
exists for all $k\in\mathbb{N}$.
\end{theorem}

Often, an explicit limit for these normalised Betti numbers can be determined in terms of $\ell^2$-Betti numbers. We will not go into this theory very deeply in this paper and refer the interested reader to for instance \cite{Lueck_Book} or \cite{Kammeyer} for more information. 

If $\Gamma$ is a group and $X$ is a finite $\Gamma$-CW complex, then we will write $b^{(2)}_k(X;\Gamma)$ for the $k$\ts{th} $\ell^2$-Betti number of the pair $(X,\Gamma)$.

We will rely on the L\"uck approximation theorem \cite{Lueck} (see also \cite[Theorem 5.26]{Kammeyer}). If $\Gamma$ is a group and $\Gamma_i\vartriangleleft \Gamma$, $i\in \mathbb{N}$ are such that
$$
[\Gamma:\Gamma_i] < \infty \quad \text{and} \quad \Gamma_{i+1} < \Gamma_i, \; i\in\mathbb{N},
$$
then we call $(\Gamma_i)_i$ a \emph{chain of finite index normal subgroups of }$\Gamma$.

\begin{theorem}[L\"uck approximation theorem] \label{thm_Lueck}
Let $\Gamma$ be a group and $X$ be a finite free $\Gamma$-CW complex. Moreover, let $(\Gamma_i)_i$ be a chain of finite index normal subgroups of $\Gamma$ and set
$$
\Theta = \bigcap_{i\in\mathbb{N}}\Gamma_i.
$$
Then
$$
\lim_{i\to\infty} \frac{b_k(X/\Gamma_i)}{[\Gamma:\Gamma_i]} = b^{(2)}(\Theta\backslash X;\; \Gamma/\Theta).
$$
\end{theorem}

In order to prove convergence of Betti numbers we are after (Theorem \ref{thm_main3}(c)), we will use the approximation theorems of Elek and L\"uck to deduce the following lemma. Like Lemma \ref{lem_IRSfixedpoints}, this lemma is probably well known but, as far as we know, not available in the literature in this form, so we will provide a proof.

\begin{lemma}\label{lem_betticonvergence} Let $\Gamma$ be a group that admits a finite simplicial complex $X$ as a classifying space. Set
$$
\mu_n = \sum_{\substack{H\;< \; \Gamma  \\ [\Gamma:H]\;=\;n}} \delta_H.
$$
If there exists a normal subgroup $N\vartriangleleft \Gamma $ such that $\Gamma/N$ is residually finite and
$$
\mu_n \stackrel{n\to\infty}{\longrightarrow} \delta_N
$$
in $\mathrm{IRS}(\Gamma)$. Then for every $\varepsilon>0$ and every $k\in\mathbb{N}$,
$$
\mu_n\left( \left|\frac{b_k(H)}{n} - b_k^{(2)}(N\backslash \widetilde{X}; \Gamma/N)\right| < \varepsilon \right) \stackrel{n\to\infty}{\longrightarrow} 1,
$$
where $\widetilde{X}$ denotes the universal cover of $X$.
\end{lemma}

\begin{proof}
Recall that $V(X)$ denotes the set of $0$-cells of $X$ and write $D$ for the maximal degree among these $0$-cells. Fix a choice of $0$-cell $x_0\in V(X)$, to obtain an identification $\Gamma\simeq \pi_1(X,x_0)$ and denote the measure on $\mathcal{K}_D$ induced by $\mu_n$ by $\nu_n \in \mathcal{BS}(\mathcal{K}_D)$. Finally, we will let $(Z,z_0)\to (X,x_0)$ denote the pointed cover corresponding to $N$.

For $g\in K\subset \Gamma$, where $K$ is a conjugacy class, $Z_K(H)$ equals the number of lifts of $x_0$ at which the loop in $X$ corresponding to $g$ lifts to a closed loop. 

Now consider the set $W_R$ of all $g\in \Gamma$ that have translation distance at most $R$ on the universal cover $\widetilde{X}$. This set consists of a finite number of conjugacy classes.

If $H<\Gamma$ is such that $[\Gamma:H] = n$ and, as $n\to\infty$,
\begin{equation}\label{eq_fparelittle-o-n}
\left\{\begin{array}{ll}
 Z_K(H) = o(n) & \text{if } K\not\subset N \cap W_R \\
 n-Z_K(H) = o(n) & \text{if } K\subset N \cap W_R
 \end{array}
\right.
\end{equation} 
then the number of lifts $y$ in the cover of $X$ corresponding to $H$, around which the $R$-ball $B_R(y)$ is not isometric to the $R$-ball $B_R(z_0)$ around $z_0\in Z$ is $o(n)$ (this uses that $W_R$ consists of finitely many conjugacy classes). 

Lemma \ref{lem_IRSfixedpoints} tells us that for any finite set of conjugacy classes, \eqref{eq_fparelittle-o-n} is satisfied with asymptotic $\mu_n$-probability $1$. So we obtain that for every $R, \varepsilon > 0$ 
$$
\nu_n\Big( \left\{[Y,y];\; \frac{|\{v\in V(Y)\text{ a lift of }x_0;\;B_R(v) \simeq B_R(z_0) \}|}{n}  > 1- \varepsilon \right\}  \Big) \stackrel{n\to\infty}{\longrightarrow} 1.
$$
Now, since $V(X)$ is finite we can repeat the argument finitely many times and obtain that for each $R>0$ there is a finite list $B_1,\ldots,B_L$ of finite simplicial complexes and a finite list of densities $\rho_1,\ldots,\rho_L >0$ such that 
$$
\nu_n\Big( \left\{[Y,y];\;\forall i: \; \left|\frac{|\{v\in V(Y);\;B_R(v) \simeq B_i \}|}{n} -\rho_i \right|  <  \varepsilon \right\}  \Big) \stackrel{n\to\infty}{\longrightarrow} 1.
$$
So, by Theorem \ref{thm_Elek}, for every $\varepsilon>0$ there exists a $\delta>0$ such that if we fix any finite pointed complex $[Q,q]\in\mathcal{K}_D$ that satisfies $ \left|\frac{|\{v\in V(Q);\;B_R(v) \simeq B_i \}|}{n} -\rho_i \right|  <  \delta $ for $i=1,\ldots,L$, then,
\begin{equation}\label{eq_betticlose}
\nu_n\Big( \left\{[Y,y];\;\forall i: \; \left|\frac{b_k(Y)}{n}  - \frac{b_k(Q)}{|V(Q)|} \right|  <  \varepsilon \right\}  \Big) \stackrel{n\to\infty}{\longrightarrow} 1 \quad \text{for all }k\in\mathbb{N}.
\end{equation}
Using the fact that $\Gamma/N$ is residually finite, we can find a chain of normal subgroups $H_i \vartriangleleft \Gamma/N$ of finite index such that $\cap_i H_i = \{e\}$. We lift this sequence of subgroups to a sequence $\widetilde{H}_i \vartriangleleft \Gamma$ and obtain a sequence of pointed covers $(Q_i,q_i) \to (X,x_0)$. Now, if we set
$$
\eta_i = \frac{1}{[\Gamma:\widetilde{H}_i]} \sum_{u \in (\Gamma/\widetilde{H}_i)\cdot q_i} \delta_{[Q_i,u]} \in \mathcal{BS}(\mathcal{K}_D),
$$
then $\eta_i \stackrel{i\to\infty}{\longrightarrow} \delta_N$ by construction. So, for \eqref{eq_betticlose}, we can take a $(Q_i,q_i)$ deep in the sequence we just constructed. Moreover, by Theorem \ref{thm_Lueck} we have
$$
\frac{b_k(Q_i)}{n} \approx  b_k^{(2)}(N\backslash \widetilde{X}; \Gamma/N),
$$
which finishes the proof.
\end{proof}

\section{Subgroup growth}\label{sec_subgroupgrowth}

Our first objective is to prove Theorem \ref{thm_main1} -- the asymptotic subgroup growth of our groups $\Gamma_{p_1,\ldots,p_m}$. Note that this follows from the following theorem together with Theorems \ref{hn_cyclic} and \ref{thm_muller-transitive} and  Lemma \ref{subgp_trans}.

\begin{theorem}\label{thm_subgroup_growth} Let $p_1,\ldots,p_m \in \mathbb{N}_{>1}$ such that $\sum_{i=1}^m \frac{1}{p_i} < m-1$. Then, as $n\to\infty$,
	$$a_n (\Gamma_{p_1, \cdots, p_m}) \sim a_n(C_{p_1}\ast \cdots \ast C_{p_m}).$$
\end{theorem}

\begin{proof}
	We have already mentioned that the group $\Gamma_{p_1, \cdots, p_m}$ is a central extension
	\[ 1\longrightarrow \mathbb{Z}  \longrightarrow \Gamma_{p_1,\ldots,p_m} \stackrel{\Phi_{p_1,\ldots,p_m}}{\longrightarrow} C_{p_1} * \cdots * C_{p_m} \longrightarrow 1\]
	where $\Phi_{p_1,\ldots,p_m}$ sends $x_j$ to the generator of $C_{p_j}$. 
	Using this alongside Proposition \ref{prop_subgp_quotient}, implies that
	
	\begin{multline*}
	 1 \leq 
	\frac{a_n (\Gamma_{p_1, \cdots, p_m})}{a_n(C_{p_1}\ast \cdots \ast C_{p_m})} \leq 
	\frac{\sum_{r|n} a_{n/r}(C_{p_1}\ast \cdots \ast C_{p_m})r^{m}}{a_n(C_{p_1}\ast \cdots \ast C_{p_m})} \\ = 
	 1 + \frac{\sum_{\substack{r|n \\ r>1}} a_{n/r}(C_{p_1}\ast \cdots \ast C_{p_m})r^{m}}{a_n(C_{p_1}\ast \cdots \ast C_{p_m})} 
	 \end{multline*}
	
	Take $r_n$ to be the divisor of $n$ with $1 < r_n \leq n$, that maximises $a_{n/r}(C_{p_1}\ast \cdots \ast C_{p_m})d^{m}$. Since the number of divisors of $n$ is certainly bounded above by $n$, we see 
	$$\frac{\sum_{\substack{r|n \\ r>1}} a_{n/r}(C_{p_1}\ast \cdots \ast C_{p_m})r^{m}}{a_n(C_{p_1}\ast \cdots \ast C_{p_m})} \leq
	n \frac{a_{n/r_n}(C_{p_1}\ast \cdots \ast C_{p_m})r_n^{m}}{a_n(C_{p_1}\ast \cdots \ast C_{p_m})}.$$
	
	Using the growth rates of $a_n(C_{p_1}\ast \cdots \ast C_{p_m})$ which are obtained from Theorem \ref{hn_cyclic} and Theorem \ref{thm_muller-transitive}, we get

	\begin{multline*}
	\lim_{n\to\infty}\left(nr_n^{m}\frac{a_{n/r_n}(C_{p_1}\ast \cdots \ast C_{p_m})}{a_n(C_{p_1}\ast \cdots \ast C_{p_m})}\right) \\
	=
	\lim_{n\to\infty}\left(
	\exp \left( \sum_{i=1}^{m} \sum_{\substack{d|p_i\\ d < p_i}} \frac{1}{d} \left( \left( \frac{n}{r_n}\right)^{\frac{d}{p_i}} - n^{\frac{d}{p_i}} \right) \right)
	n\left(\frac{(n/r_n)!}{n!}\right)^{\sum\limits_{i=1}^{m}\alpha_{p_1,\cdots, p_m}}
	r_n^{\beta_{p_1,\cdots, p_m}}
	\right) \\
	\leq \lim_{n\to\infty} \left(
	n\left(\frac{(n/r_n)!}{n!}\right)^{\alpha_{p_1,\cdots, p_m}}
	r_n^{\beta_{p_1,\cdots, p_m}}
	\right)\\
	\leq \lim_{n\to\infty} \left(
	\left(\frac{(n/2)!}{n!}\right)^{\alpha_{p_1,\cdots, p_m}}
	n^{1+\beta_{p_1,\cdots, p_m}}
	\right)	= 0,
	\end{multline*}
	where we have used $2\leq r_n \leq n$, and denoted $\alpha_{p_1,\cdots, p_m} = \sum_{i=1}^{m}(1 - \frac{1}{p_i}) -1$ and $\beta_{p_1,\cdots, p_m} = m-1 + \frac{1}{2}(\sum_{i=1}^{m}(1 - \frac{1}{p_i}))$.
	
	Hence we have shown
	$$ {a_n (\Gamma_{p_1, \cdots, p_m})} \sim {a_n(C_{p_1}\ast \cdots \ast C_{p_m})},$$
as $n\to\infty$.
\end{proof}

As an immediate consequence of this alongside Lemma \ref{subgp_trans}, we also find that 
\[
t_n(C_{p_1} \ast \cdots \ast C_{p_m}) \sim t_n(\Gamma_{p_1,\ldots, p_m})\]
as $n\to\infty$. Moreover, we obtain:
\begin{corollary}\label{cor_hn_tn}  Let $p_1,\ldots,p_m \in \mathbb{N}_{>1}$ such that $\sum_{i=1}^m \frac{1}{p_i} < m-1$. Then
\[
h_n(\Gamma_{p_1,\ldots,p_m}) \sim t_n(\Gamma_{p_1,\ldots,p_m})
\]
as $n\to\infty$.
\end{corollary}

Before we prove this corollary, we note that together with Lemma \ref{subgp_trans} and Theorem \ref{thm_main1}, it implies Theorem \ref{thm_main2}.

\begin{proof}
Write $\Gamma=\Gamma_{p_1,\ldots,p_m}$ and $\chi = m-1-\sum_{i=1}^m \frac{1}{p_i}$. Fix any $\alpha<\chi$. Our goal will be to prove that there exists a constant $C>0$ such that
\[
\frac{h_n(\Gamma)}{(n-1)!} \leq \left(1 + \frac{C}{n^\alpha}\right) \cdot a_n(\Gamma)
\]
for all $n\geq 1$. Observe that this, combined with the fact that $\frac{h_n(\Gamma)}{(n-1)!}\geq a_n(\Gamma)$ and Lemma \ref{subgp_trans}, is sufficient to prove the corollary.

We will prove our claim by induction. Let us first consider the induction step. We will use the fact (see \cite{LubotzkySegal} Corollary 1.1.4) that
\[
\frac{h_n(\Gamma)}{(n-1)!} = 
a_n(\Gamma) +
\sum_{k=1}^{n-1} \frac{h_{n-k}(\Gamma)}{(n-k)!}\;  a_k(\Gamma).
\]
Using the induction hypothesis, we get
\begin{multline*}
\frac{h_n(\Gamma)}{(n-1)!} \leq
a_n(\Gamma) +
\sum_{k=1}^{n-1} \left(1 + \frac{C}{(n-k)^\alpha}\right)\; a_{n-k}(\Gamma)\;  a_k(\Gamma)  \\
\leq a_n(\Gamma) + (1+C)\cdot \sum_{k=1}^{n-1} a_{n-k}(\Gamma)\;  a_k(\Gamma).
\end{multline*}
Theorem \ref{thm_subgroup_growth}, together with Lemma \ref{lem_rapidgrowth},
\[
\frac{h_n(\Gamma)}{(n-1)!} \leq a_n(\Gamma) \cdot \left(1+\frac{C'\cdot (1+C)}{n^\chi}\right) =  a_n(\Gamma) \cdot \left(1+\frac{C}{n^\alpha} \frac{C'\cdot (1+C)}{C\cdot n^{\chi-\alpha}}\right)
\]
for some uniform $C'>0$. Now, we observe that, for fixed $C>0$, the factor $\frac{C'\cdot (1+C)}{C\cdot n^{\chi-\alpha}}$ can be made smaller than $1$ by increasing $n$, meaning that the induction step works after a certain $n_0$. To get the base case to hold, we need to increase $C$, which only decreases the factor $\frac{C'\cdot (1+C)}{C\cdot n^{\chi-\alpha}}$, thus proving the corollary.
\end{proof}

\section{Random subgroups and covers}\label{sec_random}

In this section we will study the properties of random index $n$ subgroups of $\Gamma_{p_1,\ldots,p_m}$ and random degree $n$ covers of torus knot complements.

The basic idea is to prove that a random index $n$ subgroup of $C_{p_1}*\cdots * C_{p_m}$ (as an element of $\mathrm{IRS}(C_{p_1}*\cdots*C_{p_m})$) converges to the trivial subgroup. This, together with Theorem \ref{thm_main2} will then imply that a random index $n$ subgroup of $\Gamma_{p_1,\ldots,p_m}$ converges to $L_{p_1,\ldots,p_m}$. Both of these results will be quantitative in the sense that we have control over the number of conjugacy classes a given conjugacy class of either $C_{p_1}*\cdots * C_{p_m}$ or  $\Gamma_{p_1,\ldots,p_m}$  lifts to in a random index $n$ subgroup (Theorem \ref{thm_main3}(a)). This then immediately implies the fact that a random degree $n$ cover of $X_{p_1,\ldots,p_m}$ Benjamini--Schramm converges to $X_{p_1,\ldots,p_m}^\Phi$. Combined with Lemma \ref{lem_betticonvergence}, this convergence implies our result on Betti numbers.

\subsection{Set-up}

Recall that, if $\Gamma$ is a group and $n\in\mathbb{N}$, then
$$
\mathcal{A}_n(\Gamma) = \{H<\Gamma;\; [\Gamma:H] = n\},
$$
and that
$$
Z_K: \mathcal{A}_n(\Gamma)\to\mathbb{N}
$$
counts the number of conjugacy classes that $K$ splits into, i.e.
$$
Z_K(H) = | (K\cap H) / H |
$$
where $H$ acts on $K\cap H$ by conjugation. Finally, if we fix any $g\in K$ and $\varphi:\Gamma \to S_n$ is a transitive homomorphism corresponding to $H$ (cf. Proposition \ref{subgp_trans}), then
$$
Z_K(H) = |\{ j\in \{1,\ldots,n\};\;\varphi(g)\cdot j = j\}|.
$$
Our main goal now is to determine the asymptotic behaviour of the distribution of these random variables as $n\to\infty$.

\subsection{Fixed point statistics of infinite order elements of $C_{p_1}*\cdots*C_{p_m}$}

Our first step is to enlarge our probability space and prove our results there. Concretely, the expression for $Z_K$ in terms of fixed points is well-defined for any homomorphism, not just for transitive ones. As such, we can interpret $Z_K$ as a random variable
$$
Z_K : \Hom(C_{p_1}*\cdots*C_{p_m},S_n) \to \mathbb{N}
$$ 
as well, where we equip $\Hom(C_{p_1}*\cdots*C_{p_m},S_n)$ with the uniform measure $\mathbb{P}^{\Hom}_n$. We will denote the expected value with respect to this measure by $\mathbb{E}^{\Hom}_n$.

It will turn out that for conjugacy classes $K$ of elements of infinite order, $Z_K$ will limit to a Poisson distributed random variable. In this section we will work this out, starting with some set-up.

Suppose that $g\in C_{p_1}*\cdots*C_{p_m}$ is of infinite order and that $K$ is its conjugacy class. Write
$$
g = x_{j_1}^{s_1} \cdots x_{j_l}^{s_l}\;, \quad i=1,\ldots,r.
$$
as a reduced word in the generators $x_1,\ldots,x_m$ of $C_{p_1}*\cdots*C_{p_m}$. By potentially changing the conjugate, we may also assume that the word is cyclically reduced. 

Now, if we want $v\in \{1,\ldots,n\}$ to be a fixed point of $\varphi(g)$ for some $\varphi\in\Hom(C_{p_1}*\cdots*C_{p_m},S_n)$, then there need to be sequences $(w_{t,0}\; w_{t,1}\;\ldots w_{t,s_t})$, for $t=1,\ldots,l$, such that
\begin{equation}\label{eq_phisatisfiesw}
\left\{ \begin{array}{ll}
\varphi(x_{j_t})(w_{t,q}) = w_{t,q-1}, &  q=1,\ldots,s_t\\[2mm]
w_{t,s_t} = w_{t+1,0} & t=1,\ldots,l-1 \\[2mm]
w_{1,0}= w_{l,s_l} =  v. & 
 \end{array}
 \right.
\end{equation}
In other words, if we want $v$ to be a fixed point of $g$, then certain sequences (for which there are many choices) need to appear in the disjoint cycle decompositions of the images of the generators $x_1,\ldots,x_m$. Figure \ref{pic_cycle} gives an example of the situation. Note that some of the labels in these sequences may coincide. 

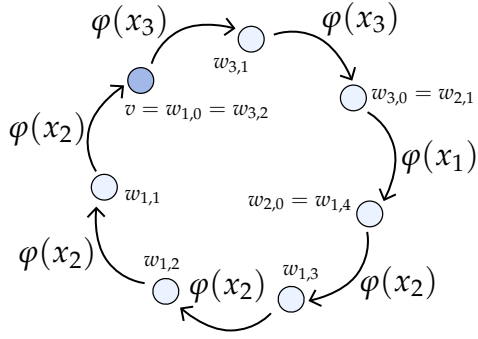
\begin{figure}
\begin{center}
\begin{overpic}{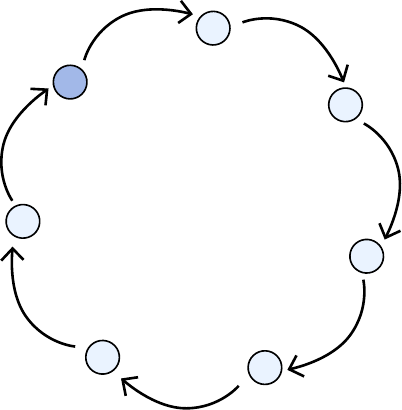}
\put(12,69) {{\tiny $v=w_{1,0}=w_{3,2}$}}
\put(1,96) {$\varphi(x_3)$}
\put(40,84) {{\tiny $w_{3,1}$}}
\put(74,97) {$\varphi(x_3)$}
\put(90,74) {{\tiny $w_{3,0} = w_{2,1}$}}
\put(99,53) {$\varphi(x_1)$}
\put(51,40) {{\tiny $w_{2,0}=w_{1,4}$}}
\put(86,13) {$\varphi(x_2)$}
\put(62,18) {{\tiny $w_{1,3}$}}
\put(32,12) {$\varphi(x_2)$}
\put(18,21) {{\tiny $w_{1,2}$}}
\put(-22,22) {$\varphi(x_2)$}
\put(12,42) {{\tiny $w_{1,1}$}}
\put(-25,62) {$\varphi(x_2)$}
\end{overpic}
\caption{$v$ is a fixed point of $\varphi(x_2^4x_1x_3^2)$.}\label{pic_cycle}
\end{center}
\end{figure}

These sequences naturally correspond to labelled graphs. The vertices are labelled by the numbers $w_{t,i}$, which we connect with edges labelled by the generators $x_{j_i}$ according to the conditions in \eqref{eq_phisatisfiesw}. We will say that $\varphi$ \emph{satisfies} this labelled graph. This allows us to decompose the variable $Z_K$ into a finite sum
\begin{equation}\label{eq_decompZK}
Z_K = \sum_{G} \mu(G)\cdot Y_G
\end{equation}
where 
\begin{itemize}
\item the sum is over finite directed graphs $G$ such that
\begin{itemize}
\item the edges are labelled with generators $x_j$,
\item $G$ contains a directed circuit that runs through all of its edges a finite number of times in such a way that the word spelled out by the generators on these edges (in the order given by the circuit) is $g$ 
\item $\mu(G)$ is the number of vertices at which a circuit of the form described above starts.
\end{itemize}
Note that these graphs $G$ do \emph{not} carry vertex labels. We will call such graphs \emph{$K$-graphs}.
\item $Y_G(\varphi)$ counts the number of labelled copies of $G$ that are satisfied by $\varphi$.
\end{itemize}

It will turn out that, asymptotically, the sum above is dominated by cycles. That is, connected $2$-regular graphs. Note that if $g=g_0^k$ is a non-trivial power of a primitive element $g_0$, the sum contains terms corresponding to cycles of length $d\cdot \ell_0(K)$ for every divisor $d$ of $k$, where $\ell_0(K)$ is the minimal word length (with respect to the generating set $x_1,\ldots,x_m$) of the primitive element $g_0$. These will all contribute to the asymptotic behaviour of $Z_K$. Note that when $G$ is a cycle,
$$
\mu(G) = \frac{\#(G)}{\ell_0(K)}
$$
where $\#(G)$ is the number of vertices in $G$ and $\ell_0(K)$ is again the minimal word length of a primitive element $g_0 \in C_{p_1}*\cdots*C_{p_m}$ for which there exists $k\in\mathbb{N}$ with $g_0^k \in K$ \footnote{note that $\ell_0(K)$ divides $\#(G)$.}.

We have:

\begin{theorem}\label{thm_factorialmomentsY} Let $p_1,\ldots,p_m\in \mathbb{N}$ and let $K_1,\ldots, K_r \subset C_{p_1}*\cdots*C_{p_m}$ be conjugacy classes out of which no pair have a common root and none of which contains elements of finite order. Finally, let $\{G_{ij}\}_{j\in J_i}$ be a finite set of distinct $K_i$-graphs for all $i=1,\ldots,r$. Then 
\begin{itemize}
\item If one of the graphs $G_{lm}$ is not a cycle and $k_{ij} \in \mathbb{N}$, $j\in J_i$, $i=1,\ldots, r$ is a tuple so that $k_{lm}>0$, we have
\[
\mathbb{E}_{n}^{\Hom}\left[\prod_{i,j}\left(Y_{G_{ij}}\right)_{k_{ij}}\right] = O\left(n^{-1}\right)
\]
as $n\to\infty$.
\item if not, then we have
$$
\mathbb{E}_{n}^{\Hom}\left[\prod_{i,j}\left(Y_{G_{ij}}\right)_{k_{ij}}\right] = \prod_{i,j} \frac{1}{\mu(G_{ij})^{k_{ij}}} + O\left(n^{-1}\right)
$$
as $n\to\infty$.
\end{itemize}
\end{theorem}

The implied constants in the two bounds in the theorem depend both on the graphs $\{G_{ij}\}_{i,j}$ and the tuple $(k_{ij})_{i,j}$. Before we prove the theorem, we observe that it immediately implies the following corollary, which in the case where $r=1$, was proved for certain conjugacy classes by Benaych-Georges \cite{BG1}. In fact, for these conjugacy classes, Benaych-Georges proves similar results for all cycles, and not just the fixed points, but he does not prove that the independence for different conjugacy classes.

\begin{corollary}\label{cor_poissononhom}
 Let $p_1,\ldots,p_m\in \mathbb{N}$ and let $K_1,\ldots, K_r \subset C_{p_1}*\cdots*C_{p_m}$ be conjugacy classes out of which no pair have a common root and none of which contain elements of finite order. Then, as $n\to\infty$, the vector of random variables
$$
\left( Z_{K_1},\ldots,Z_{K_r}\right): \Hom(C_{p_1}*\cdots*C_{p_m},S_n) \to \mathbb{N}^r
$$
converges jointly in distribution to a vector  
$$
\left(Z_{K_1}^\infty, \ldots, Z_{K_r}^\infty\right):\Omega \to \mathbb{N}^r
$$
of independent random variables, such that if $K_i$ is the conjugacy class of a $k^{th}$ power of a primitive element, then
$$
Z_{K_i}^\infty \sim \sum_{d|k}\; d \cdot X_{1/d},  \quad  \text{where } 
X_{1/d} \sim \mathrm{Poisson}(1/d) \text{ and }X_1,\ldots,X_{1/k} \text{ are independent}.
$$
\end{corollary}

\begin{proof}
Theorem \ref{thm_factorialmomentsY} together with Theorem \ref{thm_moments} implies that, when $\{G_{ij}\}_{j\in J_i}$ is a finite set of distinct $K_i$-graphs for all $i=1,\ldots,r$, then, as $n\to\infty$, the vector of random variables
$$
\left(Y_{G_{ij}}: \Hom(C_{p_1}*\cdots*C_{p_m},S_n) \longrightarrow \mathbb{N}\right)_{i,j}
$$
converges jointly in distribution to a vector  
$$
\left(Y_{G_{ij}}^\infty:\Omega \to \mathbb{N}\right)_{i,j}
$$
of independent random variables, such that:
\begin{itemize}
\item if $G_{ij}$ is a cycle, then
$$
Y_{G_{ij}}^\infty \sim \mathrm{Poisson}\left(\frac{1}{\mu(G_{ij})}\right).
$$
\item if not, then $Y_{G_{ij}}^\infty$ is constant and equal to $0$.
\end{itemize}
The decomposition in \eqref{eq_decompZK} now implies the result.
\end{proof}

\begin{proof}[Proof of Theorem \ref{thm_factorialmomentsY}] We will write $\Lambda=C_{p_1}*\cdots*C_{p_m}$ and $\mathcal{H}_n(\Lambda) = \Hom(\Lambda,S_n)$. 

Observe that the random variable 
$$
\prod_{i,j}\left(Y_{G_{ij}}\right)_{k_{ij}}:\mathcal{H}_n(\Lambda) \to \mathbb{N}
$$ 
counts tuples $(F_{ij})_{i,j}$ of vertex labelled copies $F_{ij}$ of $G_{ij}$ that are satisfied by $\varphi(g_i)$. As such, we may write
$$
\mathbb{E}_n^{\Hom}\left[\prod_{i,j}\left(Y_{G_{ij}}\right)_{k_{ij}}\right] = \sum_{\alpha \in A(n)} \mathbb{E}_n[\ind_\alpha],
$$
where 
$$
A(n) = \left\{ \alpha=(\alpha_{ij})_{i,j};\; \begin{array}{c}\alpha_{ij} \text{ a }k_{ij}\text{-tuple of }K_i\text{-graphs, isomorphic to }G_{ij}\text{, whose} \\ 
\text{vertices are labelled with elements of }\{1,\ldots.n\}\end{array}\right\}
$$
and
$$
\ind_\alpha :\mathcal{H}_n(\Lambda) \to \{0,1\}
$$
satisfies $\ind_\alpha(\varphi) = 1$ if and only if $\varphi$ satisfies all the labelled graphs corresponding to $\alpha$. Note that many of these indicators are constant $0$ functions, because the combination of labels involved leads to a contradiction about the properties of $\varphi(x_j)$ for some $j\in\{1,\ldots,m\}$.

\medskip
\noindent\textit{Cycles.} We will now start with the case where all the graphs $G_{ij}$ are cycles. We will write
$$
A(n) = A_1(n) \sqcup A_2(n)
$$
where
$$
A_1(n) = \{\alpha\in A; \;\text{every label appears at most once in }\alpha\}
$$
and
$$
\quad A_2(n) = A(n)\setminus A_1(n).
$$
Note that within a given cycle, the labels are already disjoint, so in the labelled tuples in $A_2(n)$, labels corresponding to different cycle (that are potentially copies of the same cycle $G_{ij}$) coincide.

We now need to prove two facts, namely
$$
\lim_{n\to\infty} \sum_{\alpha\in A_1(n)} \mathbb{E}^{\Hom}_n[\ind_\alpha] = \prod_{i,j} \frac{1}{\mu \left(G_{ij}\right)^{k_{ij}} }\quad \text{and} \quad  \sum_{\alpha \in A_2(n)}\mathbb{E}^{\Hom}_n[\ind_\alpha] \stackrel{n\to\infty}{=} O(n^{-1}).
$$

We start with estimating $\mathbb{E}_n^{\Hom}[\ind_\alpha]$ for $\alpha\in A_1$. Observe that
$$
\mathbb{E}_n^{\Hom}[\ind_\alpha] = \frac{ \left|\left\{\varphi\in \mathcal{H}_n(\Lambda);\; \varphi \text{ satisfies } \alpha \right\}\right|}{h_n(\Lambda)}.
$$
In order to count the numerator on the right hand side, we need to count the number of ways to complete the information given in $\alpha$ to a homomorphism $\Lambda \to S_n$. We do this as follows.

The labels in $\alpha$ must appear, in the order dictated by $\alpha$, as parts of cycles in a disjoint cycle decomposition of the $x_j$'s. We will decompose these labels appearing in $\alpha$ into \emph{words}. These are strings of labels in $\{1,2,\ldots,n\}$ that appear as the vertices of some path in $G_{ij}$, all of whose edges are marked with the same generator, and which is maximal with respect to this last property.

So, a choice needs to be made for the lengths of these cycles that contain these words, which words appear together in a cycle, and which other labels appear in these cycles. Once these cycles have been completed, this determines $m$ homomorphisms $C_{p_j} \to S_{D_j}$, where $D_j$ depends on the chosen cycle lengths. To complete this into a homomorphism  $C_{p_j} \to S_n$, we have the choice out of $h_{n-D_j}(C_{p_j},S_n)$ homomorphisms. This, as $n\to\infty$, gives a total of
$$
\sim \prod\limits_{j=1}^m \; \sum\limits_{\{S_1,\ldots,S_t\} \models W_j(\alpha)} \sum\limits_{\substack{d_1,\ldots,d_t|p_j \\ 
d_q \geq \sum_{w\in S_q} \ell(w) }} C(S,d)\cdot n^{\sum_q d_q -\sum_{w\in W_j(\alpha)}\ell(w)} h_{n-\sum_q d_q}(C_{p_j})
$$
ways to complete the information in $\alpha$ to a homomorphism, where
\begin{itemize}
\item $W_j(\alpha)$ is the set of words that appear in $\alpha$ and pose a condition on $\varphi(x_j)$,
\item the notation $\{S_1,\ldots,S_t\}\models W_j(\alpha)$ means that $\{S_1,\ldots,S_t\}$ forms a set partition of $W_j(\alpha)$ (these are the groups of words that are going to appear together in cycles in $\varphi(x_j)$),
\item the numbers $d_1,\ldots,d_t$ are going to be the lengths of the cycles containing the words in the sets $\{S_1,\ldots,S_t\}$,
\item $\ell(w)$ is the number of labels in a word $w$,
\item $C(S,d)$ is a combinatorial constant that counts the number of ways to distribute the words over cycles in according to $\{S_1,\ldots,S_t\}$ and $d_1,\ldots,d_t$. Moreover, if the set partition $\{S_1,\ldots,S_t\}$ consists of singletons and $d_1=d_2=\ldots =d_t = p_j$ then $C(S,d)=1$
\item and we have already made one simplification: the powers of $n$ should in reality take the form of a falling factorial. However, since we are only interested in asymptotics and all the products involved are of fixed bounded length, we replaced them by powers of $n$, whence the ``$\sim$''.
\end{itemize}
Now we notice that all the sums and products involved are finite, we may apply Theorem \ref{hn_cyclic} to single out the largest term. This implies that, as $n\to\infty$, 
\begin{align*}
\mathbb{E}_n^{\Hom}[\ind_\alpha] & \sim \prod_{j=1}^m n^{\sum_{w\in W_j(\alpha)} p_j-\ell(w)} \frac{h_{n-|W_j(\alpha)|\cdot p_j}(C_{p_j})}{h_n(C_{p_j})}  \\
& \sim \prod_{j=1}^m n^{\sum_{w\in W_j(\alpha)} p_j-\ell(w)} n^{-|W_j(\alpha)|\cdot p_j\cdot \left(1-\frac{1}{p_j}\right)}  \\
& =  \prod_{j=1}^m n^{\sum_{w\in W_j(\alpha)} 1 - \ell(w)}  \\
& =  \prod_{i,j} n^{-\#(G_{ij})\cdot k_{ij}}.
\end{align*}

Another important thing to observe is that $\mathbb{E}_n^{\Hom}[\ind_\alpha]$ is constant on $A_1(n)$: it does not depend on the labels involved. This implies that
\begin{equation}\label{eq_EA1}
\sum_{\alpha\in A_1(n)} \mathbb{E}_n^{\Hom}[\ind_\alpha] \sim |A_1(n)|\cdot  \prod_{i=r}^m n^{-|g_i|\cdot k_i}
\end{equation}
as $n\to\infty$. Moreover,
$$
|A_1(n)| = \left( \prod_{i,j}  \frac{1}{\mu(G_{ij})^{k_{ij}}} \right)\cdot n\cdot\left(n-1\left)\cdots \left(n-\left(\sum_{i,j} \#(G_{ij}) \cdot k_{ij}\right)+1\right)\right.\right.,
$$
it is the number of ways to the label the tuples of copies of $G_{ij}$ with distinct elements from $\{1,\ldots,n\}$. The factor $ \prod_{i,j}  \mu(G_{ij})^{-^k_{ij}}$ in this formula comes from the edge-label-preserving symmetries of the graph $G_{ij}$. Together with \eqref{eq_EA1}, this proves our claim that
$$
\lim_{n\to\infty} \sum_{\alpha\in A_1(n)} \mathbb{E}_n^{\Hom}[\ind_\alpha] = \prod_{i,j}  \frac{1}{\mu(G_{ij})^{k_{ij}}}.
$$

In order to prove that the other term tends to zero, we argue in a similar fashion. Indeed, we may think of the elements $\alpha \in A_2(n)$ as corresponding to different graphs, obtained by identifying the vertices of the different copies of the $G_{ij}$'s that carry the same labels. Write
$$
\sum_{\alpha\in A_2(n)}\mathbb{E}_n^{\Hom}[\ind_\alpha] = \sum_{G'} \sum_{\alpha\in A_{G'}(n)}\mathbb{E}_n^{\Hom}[\ind_\alpha]  
$$ 
where the sum is over isomorphism types types $G'$ of graphs obtained by collapsing the labelled graphs in $A_2(n)$ and $A_{G'}(n)$ consists of all $\alpha\in A_2(n)$ whose graph has isomorphism type $G'$.

Suppose $G'$ is such an isomorphism type with $v(G')$ vertices and $e(G')$ edges. Again $\mathbb{E}_n^{\Hom}[\ind_\alpha]$ is the same for all $\alpha\in A_{G'}(n)$. Moreover, with exactly the same arguments as above we have
\[
\mathbb{E}_n^{\Hom}[\ind_\alpha] \sim n^{-e(G')} \text{ as }n\to\infty \;\forall\alpha\in A_{G'}(n)\quad \text{and} \quad  |A_{G'}(n)|\leq n^{v(G')}.
\]
Because $v(G')<e(G')$ for all $G'$ appearing in the sum (this uses the no common root assumption, because without it some of the collapsed graphs might be cycles), the sum indeed tends to zero, which finishes the proof in the case of cycles.

\medskip
\noindent\textit{Other graphs.} If one of the graphs $G_{ij}$ is not a cycle, then the argument from the previous paragraph applies and we obtain
$$
\mathbb{E}_n^{\Hom}\left[\prod_{i,j}\left(Y_{G_{ij}}\right)_{k_{ij}}\right]  \stackrel{n\to\infty}{=} O\left(n^{-1}\right).
$$
\end{proof}

\subsection{Fixed point statistics of finite order elements of $C_{p_1}*\cdots*C_{p_m}$}

Given a finite cyclic group $C_p=\langle x| x^p \rangle$ of order $p$ and a divisor $d|p$, we let 
$$
Y_d:\Hom(C_p,S_n)\to\mathbb{N}
$$
denote the random variable that associates the number of $d$-cycles of $\varphi(x)$ to $\varphi \in \Hom(C_p,S_n)$.

M\"uller and Schlage-Puchta \cite{MullerPuchta2,MullerPuchta3} proved that these variables satisfy a central limit theorem. Similar results we also proved by Benaych-Georges in \cite{BG2}.

\begin{theorem}\label{thm_fin_ord_elts_sym_group}
Let $\mathcal{D} = \{ d \in \mathbb{N};\; d|p\}$.
\begin{enumerate}
\item~\label{clt_cycles} (Central limit theorem) \cite[Lemma 4]{MullerPuchta3} The normalised vector
$$
\left( \frac{Y_d-\frac{1}{d} n^{d/p}}{\frac{1}{\sqrt{d}} n^{d/2p}} \right)_{d\in\mathcal{D}\setminus\{p\}}: \Hom(C_p,S_n)\to\mathbb{N}^{\mathcal{D}\setminus\{p\}} 
$$
converges in distribution to $\otimes_{d\in\mathcal{D}\setminus\{p\}} \mathcal{N}(0,1)$.
\item~\label{clt_max_cycles} \cite[Lemma 3]{MullerPuchta2} (Central limit theorem for maximal cycles) Set $d_{\max}=\max\{ d\in \mathcal{D};\; d<p\}$. The normalised random variable
$$
\frac{\frac{1}{p}n-Y_p}{\frac{1}{\sqrt{d_{\max}}} n^{d_{\max}/2p}}: \Hom(C_p,S_n) \to \mathbb{N}
$$ 
converges in distribution to $\mathcal{N}(0,1)$.
\end{enumerate}
\end{theorem}

With this in hand, we can now determine the asymptotic behaviour of the variables $Z_K$ for conjugacy classes $K$ consisting of finite order elements.

\begin{corollary}\label{cor_cltonhom}
Let $K_1,\ldots,K_r\subset C_{p_1}*\cdots*C_{p_m}$ be distinct non-trivial conjugacy classes of elements of orders $k_1,\ldots,k_r\in\mathbb{N}$ respectively. Then
\begin{enumerate}
\item\label{expectation_finite_order} As $n\to\infty$,
$$
\mathbb{E}_n^{\Hom}\left[Z_{K_i}\right] \sim n^{1/k_i}.
$$
\item\label{clt_finite_order} and if for $i=1,\ldots,r$, $K_i$ is the conjugacy class of $x_{j_i}^{l_i}$, then the vector of random variables
$$
\left(\frac{Z_{K_1}-n^{1/k_1}}{\sqrt{l_1}\cdot n^{1/2k_1}},\ldots, \frac{Z_{K_r}-n^{1/k_r}}{\sqrt{l_r}\cdot n^{1/2k_r}} \right): \Hom(C_{p_1},\ldots,C_{p_m}, S_n) \to \mathbb{N}^r
$$
converges in distribution to $\mathcal{N}(0,1)^{\otimes r}$ as $n\to \infty$.
\end{enumerate}
\end{corollary}

\begin{proof}
$K_i$ is the conjugacy class of a power $x_j^l$ of one of the generators $x_j$ of $C_{p_1}*\cdots*C_{p_m}$. As such
$$
Z_{K_i} = \sum_{d|l} \;d \cdot Y_d^{(j)},
$$
where $Y_d^{(j)}$ counts the number of $d$-cycles of $x_j$.

By definition:
$$
\mathbb{E}^{\Hom}_n[Y_d] = \frac{1}{h_n(C_p)} \sum_{\rho \in \Hom(C_{p_j},S_n)} Y_d(\rho)  =   \frac{1}{h_n(C_p)} \sum_{|\pi| = n}
 h_{\pi}(C_p) \cdot \pi_d.
$$
Lemma \ref{lem_expprinc} gives us
$$
\frac{1}{d}x^{d} \cdot \exp\left( \sum_{i=1}^d \frac{1}{i}x^i y_i \right) = \frac{\partial}{\partial y_d} F_{C_p}(x,y) 
= \sum_\pi \frac{h_\pi(C_p)}{|\pi| !} x^{|\pi|} \left(\prod_{j\neq d}y_j^{\pi_j}\right) \cdot y^{\pi_d-1}\cdot \pi_d.
$$
Setting all the variables $y_i=1$, we obtain
$$
\frac{1}{d}x^{d} \exp\left( \sum_{i=1}^d \frac{a_i(C_p)}{i}x^i \right) =  \sum_{n=0}^\infty \frac{x^n}{n!} \sum_{|\pi|=n} h_\pi(C_p) \cdot \pi_d.
$$
Lemma \ref{lem_expprinc} also tells us that
$$
\frac{1}{d}x^{d} \exp\left( \sum_{i=1}^d \frac{a_i(C_p)}{i}x^i \right) = \sum_{n=0}^\infty\frac{1}{d}\frac{h_n(C_p)}{n!}x^{n+ d},
$$
so we obtain
$$
\mathbb{E}^{\Hom}_n[Y_d] =\frac{h_{n-d}(C_p)}{h_n(C_p)} \cdot \frac{n}{d} 
$$
This, together with Theorem \ref{hn_cyclic} implies point \ref{expectation_finite_order}.

Moreover, because of this bound on the expectation, we obtain that, in distribution 
$$
\frac{Z_{K_i}-n^{1/k_i}}{\sqrt{l_i}\cdot n^{1/2k_i}} \approx \frac{\frac{p}{k_i} Y^{(j)}_{p/k_i} - n^{1/k_i}}{\sqrt{l_i} \cdot n^{1/2k_i}} = \frac{ Y^{(j)}_{p/k_i} - \frac{k_i}{p} n^{1/k_i}}{\sqrt{\frac{k_1}{p_i}} \cdot n^{1/2k_i}}.
$$
On top of that, $k_i>1$, and if two of the conjugacy classes $K_{i_1}$ and $K_{i_2}$ correspond to different powers of the same generator $x_j$, then $Y_{p/k_{i_1}}$ and $Y_{p/k_{i_2}}$ are distinct. This means that point \ref{clt_cycles} of Theorem \ref{thm_fin_ord_elts_sym_group} applies and that, together with the fact that when $K_{i_1}$ and $K_{i_2}$ correspond to powers of different generators $x_{j_1}$ and $x_{j_2}$, $Z_{K_{i_1}}$ and $Z_{K_{i_2}}$ are independent -- a consequence of the free product structure -- implies item \ref{clt_finite_order}.
\end{proof}

\subsection{Statistics for random subgroups of $C_{p_1}*\cdots*C_{p_m}$ and Benjamini--Schramm convergence}

From the above we also obtain that $Z_K$ are asymptotically independent Poisson-distributed variables when seen as random variables on the set of index $n$ subgroups of $C_{p_1}*\cdots*C_{p_m}$.

\begin{theorem}\label{thm_statisticsforC*...*C}
Let $p_1,\ldots,p_m\in \mathbb{N}$ such that $\sum_{i=1}^m \frac{1}{p_i} < m-1$. Let $K_1,\ldots, K_r \subset C_{p_1}*\cdots*C_{p_m}$ be non-trivial conjugacy classes out of which no pair have a common root.
\begin{enumerate}
\item~\label{poisson_for_randsubgroup} If $K_1,\ldots, K_r \subset C_{p_1}*\cdots*C_{p_m}$ are conjugacy classes of infinite order elements. Then, as $n\to\infty$, the vector of random variables
$$
(Z_{K_1},\ldots,Z_{K_r}): \mathcal{A}_n(C_{p_1}*\cdots*C_{p_m}) \to \mathbb{N}^r
$$
converges jointly in distribution to a vector of 
$$
(Z_{K_1}^\infty,\ldots,Z_{K_r}^\infty):\Omega \to \mathbb{N}^r
$$
of independent random variables, such that if $K_i$ is the conjugacy class of a $k^{th}$ power of a primitive element
$$
Z_{K_i}^\infty \sim \sum_{d|k}\; d \cdot X_{1/d},
$$
where $X_{1/d} \sim \mathrm{Poisson}(1/d)$ and $X_1,\ldots,X_{1/k}$ are independent.
\item~\label{clt_for_ransubgroup} If $K_1,\ldots,K_r\subset C_{p_1}*\cdots*C_{p_m}$ are conjugacy classes of elements of orders $k_1,\ldots,k_r\in\mathbb{N}$ respectively and if for $i=1,\ldots,r$, $K_i$ is the conjugacy class of $x_{j_i}^{l_i}$, then the vector of random variables
$$
\left(\frac{Z_{K_1}-n^{1/k_1}}{\sqrt{l_1}\cdot n^{1/2k_1}},\ldots, \frac{Z_{K_r}-n^{1/k_r}}{\sqrt{l_r}\cdot n^{1/2k_r}} \right): \mathcal{A}_n(C_{p_1}*\cdots*C_{p_m})  \to \mathbb{N}^r
$$
converges in distribution to $\mathcal{N}(0,1)^{\otimes r}$ as $n\to \infty$.
\end{enumerate}
\end{theorem}

\begin{proof}
We will again write $\Lambda = C_{p_1}*\cdots*C_{p_m}$. 

We start with item \ref{poisson_for_randsubgroup}. Using the $(n-1)!$-to-$1$ correspondence between transitive permutation representations $\Gamma\to S_n$ and index $n$ subgroups of $\Gamma$ (i.e. Proposition \ref{subgp_trans}), what we need to prove is that for all $A\subset\mathbb{N}^r$,
$$
\frac{\left|\left\{\varphi\in\Hom(\Lambda,S_n); \begin{array}{c}
\varphi \text{ transitive} \\ (Z_{K_1},\ldots,Z_{K_r})(\varphi) \in A 
\end{array} \right\}\right|}{t_n(\Lambda)} \stackrel{n\to\infty}{\longrightarrow} \mathbb{P}[ (Z_{K_1}^\infty,\ldots,Z_{K_r}^\infty) \in A ].
$$
We have 
\begin{multline*}
\frac{\left|\left\{\varphi\in\Hom(\Lambda,S_n); \begin{array}{c}
\varphi \text{ transitive} \\ (Z_{K_1},\ldots,Z_{K_r})(\varphi) \in A 
\end{array} \right\}\right|}{t_n(\Lambda)} \\
\leq
\frac{\left|\left\{\varphi\in\Hom(\Lambda,S_n);  (Z_{K_1},\ldots,Z_{K_r})(\varphi) \in A 
\right\}\right|}{h_n(\Lambda)} \cdot \frac{h_n(\Lambda)}{t_n(\Lambda)} \\
 \stackrel{n\to\infty}{\longrightarrow} \mathbb{P}[ (Z_{K_1}^\infty,\ldots,Z_{K_r}^\infty) \in A ],
\end{multline*}
by Corollary \ref{cor_poissononhom} and Theorem \ref{thm_muller-transitive} (note that this uses that  $\sum_{i=1}^m \frac{1}{p_i} < m-1$).

Likewise,
\begin{multline*}
\frac{\left|\left\{\varphi\in\Hom(\Lambda,S_n); \begin{array}{c}
\varphi \text{ transitive} \\ (Z_{K_1},\ldots,Z_{K_r})(\varphi) \in A 
\end{array} \right\}\right|}{t_n(\Lambda)} \\
\geq
\frac{\left|\left\{\varphi\in\Hom(\Lambda,S_n);  (Z_{K_1},\ldots,Z_{K_r})(\varphi) \in A 
\right\}\right|}{h_n(\Lambda)} \cdot \frac{h_n(\Lambda)}{t_n(\Lambda)} + \frac{h_n(\Lambda) - t_n(\Lambda)}{t_n(\Lambda)} \\
 \stackrel{n\to\infty}{\longrightarrow} \mathbb{P}[ (Z_{K_1}^\infty,\ldots,Z_{K_r}^\infty) \in A ],
\end{multline*}
again by Corollary \ref{cor_poissononhom} and Theorem \ref{thm_muller-transitive}, which proves the result.

The proof of item \ref{clt_for_ransubgroup} is the same, except that we need to replace Corollary \ref{cor_poissononhom} by point \ref{clt_finite_order} of Corollary \ref{cor_cltonhom}.
\end{proof}

Our next goal is to use this to prove convergence of a random index $n$ subgroup of $C_{p_1}*\cdots*C_{p_m}$:

\begin{corollary}\label{cor_BSconvC*...*C}
Let $p_1,\ldots,p_m$ be such that $\sum_{i=1}^m \frac{1}{p_i}<m-1$. Then the IRS
$$
\mu_n = \frac{1}{a_n(C_{p_1}*\cdots * C_{p_m})}  \sum_{\substack{ H\;<\; C_{p_1}*\cdots * C_{p_m} \\ [C_{p_1}*\cdots * C_{p_m}:H] \; = \; n}} \delta_H
$$
converges to $\delta_{\{e\}}\in\mathrm{IRS}(C_{p_1}*\cdots * C_{p_m})$ as $n\to \infty$.
\end{corollary}

\begin{proof}
For any non-trivial conjugacy class $K\subset C_{p_1}*\cdots*C_{p_m}$, we have
$$
\mu_n(Z_K) \leq \frac{h_n(C_{p_1}*\cdots*C_{p_m})}{t_n(C_{p_1}*\cdots*C_{p_m})} \cdot \mathbb{E}^{\Hom}[Z_K] = o(n)
$$
as $n\to\infty$, by Theorem \ref{thm_muller-transitive} combined with Theorem \ref{thm_factorialmomentsY} and item  \ref{expectation_finite_order} of Corollary \ref{cor_cltonhom}. So the corollary follows from Lemma \ref{lem_IRSfixedpoints}.
\end{proof}

\subsection{Statistics for $\Gamma_{p_1,\ldots,p_m}$}

Now we are ready to prove our results on the properties of random index $n$ subgroups of $\Gamma_{p_1,\ldots,p_m}$. Let us start with the statistics of the variables $Z_K: \mathcal{A}_n(\Gamma_{p_1,\ldots,p_m})\to \mathbb{N}$. Given a sequence of random variables $X_n,Y_n:\Omega_n\to\mathbb{N}$, we will say $X_n$ and $Y_n$ are \emph{asymptotically independent as $n\to\infty$} if
$$
\lim_{n\to \infty} \mathbb{P}(X_n\in A \text{ and }Y_n\in B) - \mathbb{P}(X_n \in A)\cdot \mathbb{P}(Y_n\in B) = 0 \quad \forall A,B\subset\mathbb{N}.
$$

\begin{thmrep}{\ref*{thm_main3}(a)}
Let $p_1,\ldots,p_m \in \mathbb{N}_{>1}$ such that $\sum_{i=1}^m \frac{1}{p_i}<m-1$. Let $K_1,\ldots,K_r\subset \Gamma_{p_1,\ldots,p_m}$ be non-trivial conjugacy classes of which no pair have a common root.
\begin{enumerate}
\item If for all $g\in K_i$, for all $i=1,\ldots,r$, the image $\Phi_{p_1,\ldots,p_m}(g)$ is either trivial or of infinite order, then, as $n\to\infty$, the random variables $Z_{K_i}(H_n)$, $i=1,\ldots,r$ are asymptotically independent. Moreover,
\begin{itemize}
\item~\label{main3a_inf_order}if $K_i \subset L_{p_1,\ldots,p_m}$ then
$$
\lim_{n\to\infty} \mathbb{P}[Z_{K_i}(H_n) = n] =1
$$
\item and if $K_i \not\subset L_{p_1,\ldots,p_m}$ is the conjugacy class of the $k^{th}$ power of a primitive element then $Z_{K_i}(H_n)$ converges in distribution to a random variable
$$
Z_{K_i}^\infty \sim \sum_{d|k}\; d \cdot X_{1/d},
$$
where $X_{1/d} \sim \mathrm{Poisson}(1/d)$ and $X_1,\ldots,X_{1/k}$ are independent.
\end{itemize}
\item ~\label{main3a_fin_order}If the images of the elements of $K_i$ under $\Phi_{p_1,\ldots,p_m}$ have order $k_i\in\mathbb{N}$ for $i=1,\ldots,r$, then the vector of random variables
$$
\left(\frac{Z_{K_1}(H_n)-n^{1/k_1}}{\sqrt{l_1}\cdot n^{1/2k_1}},\ldots, \frac{Z_{K_r}(H_n)-n^{1/k_r}}{\sqrt{l_r}\cdot n^{1/2k_r}} \right)
$$
converges in distribution to a $\mathcal{N}(0,1)^{\otimes r}$-distributed random variable as $n\to \infty$. Here $l_i\in\mathbb{N}$ is such that $\Phi_{p_1,\ldots,p_m}(K_i)$ is the conjugacy class of $x_{j_i}^{l_i}$, for $i=1,\ldots,r$.
\end{enumerate}
\end{thmrep}

\begin{proof}
Let us write $\Gamma=\Gamma_{p_1,\ldots,p_m}$ and
$$
\mathcal{T}_n(\Gamma) = \{\varphi\in \Hom(\Gamma,S_n);\; \varphi(\Gamma)\curvearrowright \{1,\ldots,n\} \text{ transitively}\}.
$$
The distribution of $Z_{K_i}$ is the same on $\mathcal{T}_n(\Gamma)$ as it is on $\mathcal{A}_n(\Gamma)$. By Theorem \ref{thm_main2}, as $n\to\infty$ a typical element of $\mathcal{T}_n(\Gamma)$ factors through $\Phi_{p_1,\ldots,p_m}$. So the limiting distribution of the $Z_{K_i}$ is the same as that on
$$
\mathcal{T}_n(\Gamma)^\Phi := \{\varphi \in \mathcal{T}_n(\Gamma);\; \varphi \text{ factors through }\Phi_{p_1,\ldots,p_m}\}.
$$
 
Now, let us start with item \ref{main3a_inf_order}. If $K_i\subset L_{p_1,\ldots,p_m} = \ker(\Phi_{p_1,\ldots,p_m})$ then $Z_{K_i}$ is constant and equal to $n$ on $\mathcal{T}_n(\Gamma)^\Phi$. If $K_i \not\subset  L_{p_1,\ldots,p_m}$, then the limiting distribution of $Z_{K_i}$ on $\mathcal{T}_n(\Gamma)^\Phi$ is given by Theorem \ref{thm_statisticsforC*...*C}.\ref{poisson_for_randsubgroup}. Finally, Theorem \ref{thm_statisticsforC*...*C}.\ref{poisson_for_randsubgroup} also gives us the asymptotic independence among the $Z_{K_i}$ for $K_i \not\subset L_{p_1,\ldots,p_m}$ and the independence of the whole set follows from the fact that constant random variables are independent of any other random variable.

For finite order elements (item \ref{main3a_fin_order}), the theorem result follows from Theorem \ref{thm_statisticsforC*...*C}.\ref{clt_for_ransubgroup}.
\end{proof}

Next, we determine the limit of a random index $n$ subgroup of $\Gamma_{p_1,\ldots,p_m}$ as an IRS:
\begin{thmrep}{\ref*{thm_main3}(b)} Let $p_1,\ldots,p_m \in \mathbb{N}_{>1}$ be such that $\sum_{j=1}^m \frac{1}{p_j} < m-1$. Then the IRS 
$$
\mu_n = \frac{1}{a_n(\Gamma_{p_1,\ldots,p_m})} \sum_{\substack{H \;<\; \Gamma_{p_1,\ldots,p_m} \\ [\Gamma_{p_1,\ldots,p_m}:H] \;=\; n}} \delta_H \quad \stackrel{w^*}{\longrightarrow} \quad \delta_{L_{p_1,\ldots,p_m}}
$$
as $n\to\infty$.
\end{thmrep}

\begin{proof}
Write 
$$
\mathcal{A}_n(\Gamma_{p_1,\ldots,p_m}) = \mathcal{A}_{n,1}\sqcup\mathcal{A}_{n,2},
$$
where
$$
\mathcal{A}_{n,1} = \Phi_{p_1,\ldots,p_m}^{-1}\left(\mathcal{A}_n(C_{p_1}*\cdots*C_{p_m})\right) \quad \text{and}\quad \mathcal{A}_n(\Gamma_{p_1,\ldots,p_m}) \setminus \mathcal{A}_{n,1}
$$
If $f:\mathrm{Sub}(\Gamma_{p_1,\ldots,p_m})\to \mathbb{R}$ is a continuous function then
$$
\mu_n(f) = \frac{1}{a_n(\Gamma_{p_1,\ldots,p_m})}\sum_{H\in \mathcal{A}_n(C_{p_1}*\cdots*C_{p_m})} f(\Phi_{p_1,\ldots,p_m}^{-1}(H)) +  \frac{1}{a_n(\Gamma_{p_1,\ldots,p_m})} \sum_{H\in \mathcal{A}_{n,2}} f(H).
$$
Since $f$ is bounded and $\mu_n(\mathcal{A}_{n,2}) \to 0$ (by Theorem \ref{thm_main2}), the second term tends to $0$ as $n\to\infty$. The first term tends to $f(\ker(\Phi_{p_1,\ldots,p_m}))$ by Corollary \ref{cor_BSconvC*...*C}, which proves the theorem.
\end{proof}

Finally, we will determine the limits of the normalised Betti numbers. We have:

\begin{thmrep}{\ref*{thm_main3}(c)} Let $p_1,\ldots,p_m\in\mathbb{N}_{>1}$ be such that $\sum_{i=1}^m \frac{1}{p_i} < m-1$. For every $\varepsilon>0$ it holds that 
$$
\lim_{n\to\infty} \mu_n\left(\left|\frac{b_k(H;\mathbb{R})}{n}- b_k^{(2)}(X_{p_1,\ldots,p_m}^\Phi;\;C_{p_1}*\ldots*C_{p_m})\right| < \varepsilon \right)= 1 \quad \text{for all }k \in \mathbb{N}.
$$
Moreover,
$$
b_k^{(2)}(X_{p_1,\ldots,p_m}^\Phi;\;C_{p_1}*\ldots*C_{p_m}) = \left\{
\begin{array}{ll}
m-1 - \sum\limits_{i=1}^m \frac{1}{p_i} & \text{if } k = 1,2 \\
0 & \text{otherwise}.
\end{array}
\right. 
$$
\end{thmrep}

It follows from the fact that $\mu_n$ converges to $\delta_{L_{p_1,\ldots,p_m}}$ together with Lemma \ref{lem_betticonvergence} that the normalised Betti numbers of a random index $n$ subgroup converge to those of the cover corresponding to $L_{p_1,\ldots,p_m}$. So the only thing that we still have to do is compute these betti numbers, which is the content of the following lemma.

Recall that $\Gamma/L_{p_1,\ldots,p_m} \simeq C_{p_1}*\cdots * C_{p_m}$. We have:
\begin{lemma}\label{lem_l2_betti}
Let $X_{p_1,\ldots,p_m}$ be a classifying space for $\Gamma_{p_1,\ldots,p_m}$ and let  $X_{p_1,\ldots,p_m}^\Phi \to X_{p_1,\ldots,p_m}$ denote the cover corresponding to $L_{p_1,\ldots,p_m} \vartriangleleft \Gamma_{p_1,\ldots,p_m}$. Then
$$
b_k^{(2)}(X_{p_1,\ldots,p_m}^\Phi;\;C_{p_1}*\ldots*C_{p_m}) = \left\{
\begin{array}{ll}
m-1 - \sum\limits_{i=1}^m \frac{1}{p_i} & \text{if } k = 1,2 \\
0 & \text{otherwise}.
\end{array}
\right. 
$$
\end{lemma}

\begin{proof}
Since the $\ell^2$-Betti numbers do not depend on the choice of classifying space, we identify $\Gamma_{p_1,\ldots,p_m}$ with a lattice in $\mathrm{PSL}(2,\mathbb{R})\times \mathbb{R}$ and set $X_{p_1,\ldots,p_m} = \Gamma_{p_1.\ldots,p_m}\backslash \Big(\mathbb{H}^2\times \mathbb{R}\Big)$. This gives us an identification
$$
X_{p_1,\ldots,p_m}^\Phi =  \mathbb{H}^2\times \mathbb{R}/\mathbb{Z} .
$$
The action of $C_{p_1} *\cdots *C_{p_m}$ on $X_{p_1,\ldots,p_m}$ preserves the factors. The action on $\mathbb{R}/\mathbb{Z}$ is through the quotient $C_{p_1} *\cdots *C_{p_m} \longrightarrow C_{p_1} \times \cdots \times C_{p_m}$. The kernel $\Lambda$ of this quotient is a free group that acts trivially on $\mathbb{R}/\mathbb{Z}$. 

Because $X_{p_1,\ldots,p_m}$ is three-dimensional and $\Lambda$ is infinite,
$$b_k^{(2)}(X_{p_1,\ldots,p_m}^\Phi;\;\Lambda)=0 \quad\text{for } k\in \{0,4,5,\ldots\}
$$ 
(see for instance \cite[Theorem 1.35(8)]{Lueck_Book} or \cite[Theorem 3.18(ii)]{Kammeyer}). Moreover, since 
$$
 b_k^{(2)}(\mathbb{H}^2;\;\Lambda) = \left\{\begin{array}{ll}
 -\chi(\Lambda\backslash\mathbb{H}^2) & \text{if }k=1 \\
 0 & \text{otherwise}
 \end{array}
 \right.
$$
(see for instance \cite[Example 3.16]{Lueck_Book} or \cite[Exercise 3.3.1]{Kammeyer}) where $\chi$ denotes Euler characteristic. So, the K\"unneth formula (see for instance \cite[Theorem 1.35(4)]{Lueck_Book} or \cite[Theorem 3.18(iii)]{Kammeyer}) gives us that
$$
b_1^{(2)}(X_{p_1,\ldots,p_m}^\Phi;\;\Lambda) = b_2^{(2)}(X_{p_1,\ldots,p_m}^\Phi;\;\Lambda) = -\chi(\Lambda\backslash\mathbb{H}^2) \text{ and }
b_3^{(2)}(X_{p_1,\ldots,p_m}^\Phi;\;\Lambda) =0. 
$$
Since both orbifold Euler characteristic and $\ell^2$-Betti numbers are multiplicative with respect to finite index subgroups (see for instance \cite[Theorem 1.35(9)]{Lueck_Book} or \cite[Theorem 3.18(iv)]{Kammeyer} for the latter), the lemma follows.

\end{proof}

\begin{proof}[Proof of Theorem \ref*{thm_main3}(c)]
This is now direct from Theorem \ref*{thm_main3}(b) and Lemmas \ref{lem_l2_betti} and \ref{lem_betticonvergence}.
\end{proof}

\subsection{Random index $n$ subgroups of Fuchsian groups}\label{sec_fuchsian}

In this last section we discuss applications of our results to random subgroups of Fuchsian groups. We have:

\begin{thmrep}{\ref*{thm_main4}}
Let $\Lambda$ be a non-cocompact Fuchsian group of finite covolume.
Moreover, let $G_n<\Lambda$ denote an index $n$ subgroup, chosen uniformly at random. 
\begin{itemize}
\item[(a)] 
\begin{enumerate}
\item If $K_1,\ldots, K_r \subset \Lambda$ are conjugacy classes of infinite order elements of which no pair have a common root. Then, as $n\to\infty$, the vector of random variables
$$
(Z_{K_1}(G_n),\ldots,Z_{K_r}(G_n))
$$
converges jointly in distribution to a vector 
$$
(Z_{K_1}^\infty,\ldots,Z_{K_r}^\infty):\Omega \to \mathbb{N}^r
$$
of independent random variables, such that if $K_i$ is the conjugacy class of a $k^{th}$ power of a primitive element then
$$
Z_{K_i}^\infty \sim \sum_{d|k}\; d \cdot X_{1/d},
$$
where $X_{1/d} \sim \mathrm{Poisson}(1/d)$ and $X_1,\ldots,X_{1/k}$ are independent.
\item \cite[Lemma 4]{MullerPuchta3} If $K_1,\ldots,K_r\subset \Lambda$ are non-trivial conjugacy classes of which no pair have a common root, whose elements have orders $k_1,\ldots,k_r\in\mathbb{N}$ respectively, then the vector of random variables
$$
\left(\frac{Z_{K_1}(G_n)-n^{1/k_1}}{\sqrt{l_1}\cdot n^{1/2k_1}},\ldots, \frac{Z_{K_r}(G_n)-n^{1/k_r}}{\sqrt{l_r}\cdot n^{1/2k_r}} \right)
$$
converges in distribution to a $\mathcal{N}(0,1)^{\otimes r}$-distributed random variable as $n\to \infty$. Here $l_i\in\mathbb{N}$ is such that $K_i$ is the conjugacy class of $x_{j_i}^{l_i}$, for $i=1,\ldots,r$.
\end{enumerate}
\item[(b)] As $n\to\infty$, $G_n$ converges to the trivial group as an IRS.
\end{itemize}
\end{thmrep}

\begin{proof}[Proof sketch] First of all note that non-cocompact Fuchsian group of finite covolume are exactly groups of the form $F_r*C_{p_1}*\cdots * C_{p_m}$, with $-r+m-1 - \sum_{i=1}^m 1/p_i < 0$, where $F_r$ denotes the free groups on $r$ generators.

If $r=0$, (a) and (b) are the content of Theorem \ref{thm_statisticsforC*...*C} and Corollary \ref{cor_BSconvC*...*C} respectively. If $r>0$, the proof of Theorem \ref{thm_statisticsforC*...*C}.\ref{poisson_for_randsubgroup} needs to be adapted slightly: $r$ of the generators are now allowed to have any permutation of their image and not just permutations of a fixed order. With exactly the same strategy (and slightly easier computations, which we leave to the reader) the analogue of Theorem \ref{thm_statisticsforC*...*C}.\ref{poisson_for_randsubgroup} can now be proved (if $m=0$, much better bounds are in fact available \cite{DJPP,Nica}). In order to prove the analogue of Corollary \ref{cor_BSconvC*...*C}, the only new ingredient that is needed is that $t_n(\Gamma)/h_n(\Gamma) \to  1$. When $m=0$, this is a direct consequence of Dixon's theorem \cite{Dixon}. For the remaining cases, the proof has not been written down, but a similar strategy does the trick. Indeed, the results by by Volynets--Wilf (Theorem \ref{hn_cyclic}) together with Stirling's approximation that for $p>1$,
\[
h_n(F_r*C_{p_1}*\cdots*C_{p_m}) \sim B \cdot n^{r/2} \cdot \exp\left(\sum_{i=1}^m \sum_{d|p_i,d<p_i} \frac{1}{d}n^{d/p_i}\right) \cdot \left(\frac{n}{e}\right)^{n\left(r + \sum_{i=1}^m 1-\frac{1}{p_i}\right)},
\]
as $n\to\infty$, where $B$ is a constant depending on $(r,p_1,\ldots,p_m)$. We have
\[
1 - \frac{t_n(\Lambda)}{h_n(\Lambda)} = \sum_{k=1}^{n-1} \binom{n-1}{k-1}\frac{t_k(\Lambda) h_{n-k}(\Lambda)}{h_n(\Lambda)}
\]
(see for instance \cite[Lemma 1.1.3]{LubotzkySegal}). Combining the two, we get that there exists a constant $A>0$ such that
\[
1 - \frac{t_n(\Lambda)}{h_n(\Lambda)} \leq A \sum_{k=1}^{n-1} \binom{n}{k}^{1 -r - m - \sum\limits_{i=1}^m \frac{1}{p_i}} \exp\left(\sum_{i=1}^m\sum_{\substack{d|p_i, \\ d<p_i}}\frac{(n-k)^{d/p_i}+k^{d/p_i} -n^{d/p_i}}{d} \right) \longrightarrow 0,
\]
as $n\to\infty$, which settles the remaining cases.
\end{proof}

\appendix
\section{An elementary approach to the subgroup growth of $\Gamma_{p_1,\ldots,p_m}$} \label{app_elt_approach}

In this appendix we sketch an alternative route to Theorems \ref{thm_main1} and \ref{thm_main2}. The proof sketched below, when worked out in full, is longer than the proof we give above. But, it's more self contained and it also leads to a closed formula for $h_n(\Gamma_{p_1,\ldots,p_m})$ and hence for $a_n(\Gamma_{p_1,\ldots,p_m})$, which cannot be obtained from the proof above. On the other hand, while the large $n$ asymptotics can be derived directly from these sequences without relying on the results by M\"uller \cite{Muller_FreeProducts} on the subgroup growth of free products, this computation uses very similar methods to his (in particular the asymptotic behaviour of the coefficients of certain analytic functions due to M\"uller \cite{Muller}, Volynets \cite{Volynets} and Wilf \cite{Wilf}).

\subsection{A Closed Formula}

Our first objective is now to derive a closed formula for $h_n(\Gamma_{p_1,\ldots,p_m})$. We have:
\begin{proposition}\label{hn_closed_formula}
Let $n,p_1,\ldots,p_m \in \mathbb{N}$. Then
$$
h_n(\Gamma_{p_1,\ldots,p_m}) = n! \sum_{\substack{r_1,\ldots,r_n\geq 0 \\ \text{s.t. }\sum_l r_l\cdot l = n}} \; \prod_{\substack{1\leq l \leq n \\ \text{s.t. } r_l >0 }} (r_l!\cdot l^{r_l})^{m-1} \prod_{i=1}^m \sum_{k\in \mathcal{K}(p_i,l,r_l)} \prod_{j=1}^{p_i} \frac{1}{(l\cdot j)^{k_j} k_j!}
$$
where
$$\mathcal{K}(p,l,r) = \left\{k \in \mathbb{N}^p \left|\ \sum_{i=1}^p k_i\cdot i = r \text{ and } k_i = 0 \text{ whenever } \gcd(i\cdot l, p) \neq i  \right. \right\}.$$
\end{proposition}

The main ingredient for the formula above is the count of the number of $m$\ts{th} roots of a given permutation $\pi\in S_n$ -- i.e. the number 
$$ N_m(\pi) = \left|\left\{ \sigma\in S_n;\; \sigma^m = \pi \right\}\right|. $$
Note that this number only depends on the conjugacy class of $\pi$. The computation of $N_m(\pi)$ is a classical problem, that to the best of our knowledge has been first worked out by Pavlov \cite{Pavlov}.

Let us first introduce some notation. Recall that the conjugacy class of a permutation $\pi\in S_n$ is determined by its \emph{cycle type} -- the unordered partition of $n$ given by the lengths of the cycles in a disjoint cycle decomposition of $\pi$. In what follows the notation $1^{r_1}2^{r_2}\cdots n^{r_n}$ will denote the partition of $n$ that has $r_1$ parts of size $1$, $r_2$ parts of size $2$, et cetera. $K(1^{r_1}2^{r_2}\cdots n^{r_n}) \subset S_n$ will denote the corresponding conjugacy class. In this notation, we will often omit the sizes of which there are $0$ parts and write $i$ for $i^1$.

\begin{proposition}[Pavlov \cite{Pavlov}]\label{mth_roots} Let $m,n\in\mathbb{N}$ and $\pi \in K(1^{r_1}2^{r_2}\cdots n^{r_n}) \subset S_n$. Then
$$ N_m(\pi) = \prod_{\substack{1\leq l \leq n \\ \text{s.t. }r_l>0}} r_l!\; l^{r_l} \sum_{k\in \mathcal{K}(m,l,r_l)} \prod_{i=1}^m \frac{1}{(l \cdot i)^{k_i}k_i!}$$
where $\mathcal{K}(m,l,r)$ is as in Proposition \ref{hn_closed_formula}.
\end{proposition}

Note that there may be an $l$ such that $r_l>0$ and $\mathcal{K}(m,l,r_l)=\emptyset$. In this case, $N_m(\pi)=0$.

\begin{proof}[Proof of Proposition \ref{hn_closed_formula}]
Given a conjugacy class $K\subset S_n$, we write $N_m(K)$ for the number of roots of an element $\pi\in K$. We have
$$
h_n(\Gamma_{p_1.\ldots, p_m}) = \sum_{\substack{K\subset S_n \\ \text{a conjugacy class}}} |K| \cdot N_{p_1}(K) \cdots N_{p_m}(K).
$$
Using Proposition \ref{mth_roots} and the fact that $|K(1^{r_1}\cdots n^{r_n})| = n!/\prod_{i=1}^n i^{r_i} r_i!$ gives the formula.
\end{proof}

\subsection{Asymptotics}

Theorems \ref{thm_main1} and \ref{thm_main2} can now be derived as follows. 

First of all, it turns out that the sum in  Proposition \ref{hn_closed_formula} is dominated by a single term, namely the term corresponding to 
$$
(r_1,r_2,\ldots,r_n) = (n,0,\ldots,0).
$$
In order to prove this, we write 
$$
\tau_{p,l,r} = \sum_{k\in \mathcal{K}(p,l,r)} \prod_{j=1}^p \frac{1}{(l\cdot j)^{k_j}k_j!}
$$
so that
\begin{equation}\label{eq_hn_tau}
h_n(\Gamma_{p_1,\ldots,p_m}) = n! \sum_{\substack{r_1,\ldots,r_n \geq 0 \\ \text{s.t. } \sum_l r_l\cdot l = n}} \; \prod_{\substack{1\leq l \leq n \\ \text{s.t. } r_l>0}} (r_l! l^{r_l})^{m-1} \prod_{i=1}^m \tau_{p_i,l,r_l}
\end{equation}
by Proposition \ref{hn_closed_formula}. The crucial bound is now:
\begin{lemma}\label{tau_bounds}
Let $p, l,r \in \mathbb{N}$. Then
	$$\tau_{p,l,r} \leq  \left(\frac{1}{r \cdot l}\right)^{\frac{r}{p}} \cdot 
	\exp \left( \sum_{i|p} \frac{(r \cdot l)^{i/p}}{i\cdot l} \right).$$
\end{lemma}
\begin{proof}[Proof sketch]
We define a generating function:
	$$
	F_{p,l}(x) = \sum_{r=0}^{\infty} \tau_{p,l,r} x^r.
	$$
A direct computation gives us that
 $$F_{p,l}(x)= \prod_{i \in I_{p,l}} \exp\left(\frac{x^i}{i\cdot l}\right),$$ where $I_{p,l} = \{i \leq p \mid \gcd(i\cdot l, p)= i\}.$	
The fact that $F_{p,l}$ has non-negative coefficients implies that
$$
 \tau_{p,l,r} \leq \frac{F(x_0)}{x_0^r}
$$
for all $x_0\in (0,\infty)$. Filling this in for $x_0 = (r\cdot l)^{1/p}$ leads to the bound.
\end{proof}

Using this lemma, and the asymptotic behaviour of $\tau_{p,1,n}$ which follows from results due to M\"uller \cite{Muller}, Volynets \cite{Volynets} and Wilf \cite{Wilf}, one obtains that, the term in Proposition \ref{hn_closed_formula} we claim is dominant as $n\to\infty$  is indeed dominant.

This now first of all proves Theorem \ref{thm_main2}, because, 
the term that determines the asymptotic corresponds to maps 
$$
\varphi: \Gamma_{p_1,\ldots,p_m} = \langle x_1,\ldots x_m|\;x_1^{p_1} = \ldots = x_m^{p_m}\rangle \longrightarrow S_n
$$
such that $\varphi(x_i^{p_i})$ is the identity element in $S_n$. These are exactly the maps that factor through $\Phi_{p_1,\ldots,p_m}$.

We also obtain an asymptotic equivalent for $h_n(\Gamma_{p_1,\ldots,p_m})$ (which can also be derived from Corollary \ref{cor_hn_tn}). This, together with Lemmas \ref{subgp_trans} and \ref{lem_rapidgrowth}, then implies Theorem \ref{thm_main1}.

\section{Rapidly divergent sequences}

In two of our proofs above we need a bound on a quadratic combination of the terms of a rapidly divergent sequence. Many variations on the bound we need are known to hold. These bounds are for instance responsible for the fact that as $n\to\infty$, a random homomorphism $F_2\to S_n$ becomes transitive and the fact that, as $n\to\infty$, a random cubic graph on $n$ vertices becomes connected.

Even if the bound we present certainly isn't new, we are not aware of a statement of it in the literature. For instance, \cite[Theorem A.1.(ii)]{CMZ} comes close but does not apply to our case. The bound in \cite[Proposition 1]{Muller}, or rather its proof, does contain what we need, but is phrased in a somewhat different language. Given this, we will provide a proof.

\begin{lemma}\label{lem_rapidgrowth}
Let $a_n \in \mathbb{R}$ for all $n\in\mathbb{N}$ be such that there exist $C>0
$, $\alpha>0$ and $\beta_1,\ldots,\beta_m \in \mathbb{R}$ and $\gamma_1,\ldots,\gamma_m \in (0,1]$ such that
\[
a_n \sim C\cdot \exp\left(\alpha n \log(n) + \sum_{i=1}^m \beta_i n^{\gamma_i}\right) \quad \text{as } n\to \infty,
\]
then
\[
\sum_{k=1}^{n-1} \frac{a_k\; a_{n-k}}{a_n} = O\left(n^{-\alpha}\right)\quad \text{as } n\to \infty.
\]
\end{lemma}

\begin{proof}
Our goal is to show that the sum is dominated by its outer terms, that is, when either $k$ or $n-k$ is small. For simplicity of exposition, we will assume the coefficients $\beta_i$ are positive (as they are in our application), the general case follows from a similar argument.

We will assume that $k\leq n/2$ in what follows and deal with the other half of the sum by symmetry. Write $c_{n,k} = a_k a_{n-k}/a_n$. 

We have
\begin{multline*}
\frac{c_{n,k+1}}{c_{n,k}} = (1+o_{n,k}(1)) \cdot \exp\Bigg[ \alpha \Big( (k+1) \log(k+1) - k \log(k) \\
+ (n-k-1) \log(n-k-1) - (n-k) \log(n-k) \Big) \\
 + \sum_{i=1}^m \beta_i \cdot \Big( (k+1)^{\gamma_i} - k^{\gamma_i} 
  + (n-k-1)^{\gamma_i} - (n-k)^{\gamma_i} \Big) \Bigg].
\end{multline*}
Using the fact that for $x>0$ and $\gamma \in (0,1]$
\[
\gamma \cdot (x+1)^{\gamma-1} \leq (x+1)^\gamma - x^\gamma = \int_x^{x+1} \gamma \cdot y^{\gamma-1} dy \leq \gamma\cdot x^{\gamma-1},
\]
and that for $x\geq 1$
\[
\log(x) + 1 \leq (x+1)\log(x+1) - x \log(x) = \int_x^{x+1}(\log(y) + 1)dy \leq \log(x+1) + 1,
\]
we obtain that
\begin{multline*}
\frac{c_{n,k+1}}{c_{n,k}} \leq (1+o_{n,k}(1)) \cdot \exp\Bigg[ \alpha \Big( \log(k+1) - \log(n-k-1) \Big) \\
 + \sum_{i=1}^m \beta_i \cdot \gamma_i \cdot \Big( k^{\gamma_i-1} - (n-k)^{\gamma_i-1} \Big) \Bigg].
\end{multline*}
We observe that the second term in the exponential is uniformly bounded, from which we conclude that there exists a uniform $L \in \mathbb{N}$ such that 
\[
\frac{c_{n,k+1}}{c_{n,k}}  \leq 1
\] 
for all $L \leq k \leq n/2-L$, whenever $n$ is large enough. From a similar computation we obtain that 
\[
c_{n,k} \leq c_{n,\lfloor n/2-L \rfloor}
\]
for all $n/2-L < k \leq n/2$. In other words, the sum above is dominated by its first and last $L$ terms.

One now computes that for $k \in \mathbb{N}$ fixed,
\[
c_{n,k} = O\left( n^{-\alpha k} \right) \quad \text{as } n\to\infty.
\]
This means that if we set $M = \max\{L,\lceil 2/\alpha \rceil\}$, then
\[
\sum_{k=1}^{n-1} \frac{a_k\; a_{n-k}}{a_n} \leq c_{n,M}\cdot n + 2\cdot \sum_{k=1}^{M- 1} c_{n,k} = O\left(n^{-\alpha}\right) \quad \text{as } n\to\infty,
\]
thus proving the lemma.
\end{proof}


\bibliographystyle{alpha}
\bibliography{bib}

\end{document}